\RequirePackage[l2tabu, orthodox]{nag}
\documentclass[12pt]{amsart}
\pdfoutput=1
\usepackage{fullpage,url,amssymb,enumerate,colonequals,pdfsync,float,mathtools,appendix,float,amsmath}

\usepackage{rotating}

\graphicspath{{./figs/}}
\usepackage{tikz-cd,graphicx}
\usepackage{pinlabel}
\usepackage{epstopdf}

\makeatletter
\g@addto@macro\bfseries{\boldmath} % This makes math in section titles bold.
\makeatother

\usepackage{microtype}

\usepackage{backref}
\usepackage[alphabetic,backrefs,lite,nobysame]{amsrefs} % for bibliography

\usepackage{color}

\usepackage[framemethod=tikz]{mdframed}

\numberwithin{figure}{section}

\title{Barcode Embeddings for Metric Graphs}
\author{Steve Oudot, Elchanan Solomon}
\date{\today}

\begin{document}
\newtheorem{theorem}{Theorem}[section]
\newtheorem{claim}[theorem]{Claim}
\newtheorem{conj}[theorem]{Conjecture}
\newtheorem{lemma}[theorem]{Lemma}
\newtheorem*{exercise}{Exercise}
\newtheorem{example}{Example}
\newtheorem{prop}[theorem]{Proposition}
\newtheorem{problem}{Problem}
\newtheorem*{theorem*}{Theorem}
\newtheorem*{prop*}{Proposition}
\newtheorem*{lemma*}{Lemma}
\newtheorem*{claim*}{Claim}
\newtheorem{corollary}[theorem]{Corollary}
\newtheorem*{corollary*}{Corollary}
\theoremstyle{remark}
\newtheorem*{notation}{Notation}
\newtheorem{remark}[theorem]{Remark}
\newtheorem{observation}[theorem]{Observation}
\theoremstyle{definition}
\newtheorem{counterexample}[theorem]{Counterexample}
\newtheorem{dictionary}[theorem]{Dictionary}
\newtheorem{definition}[theorem]{Definition}
\newtheorem*{moral}{Moral}
	\newcommand{\lp}[1]{\lVert #1 \rVert_{L^p}}
	\newcommand{\id}{\operatorname{id}}
	\newcommand{\norm}[1]{\left\lVert#1\right\rVert}
	\newcommand{\supp}[1]{\operatorname{supp}(#1)}
	
	\newcommand{\C}{\mathbb{C}}
\newcommand{\R}{\mathbb{R}}
\newcommand{\Rop}{\mathbb{R}^{\text{op}}}
\newcommand{\Rext}{\mathbb{R}_{\text{Ext}}}
\newcommand{\N}{\mathbb{N}}
\newcommand{\Z}{\mathbb{Z}}
\newcommand{\X}{\mathbb{X}}
\newcommand{\Y}{\mathbb{Y}}
\newcommand{\Reeb}{\mathrm{R}}
\newcommand{\Ord}{\mathrm{Ord}}
\newcommand{\Ext}{\mathrm{Ext}}
\newcommand{\Rel}{\mathrm{Rel}}
\newcommand{\Dg}{\mathrm{Dg}}
\newcommand{\Mapper}{\mathrm{M}}
\newcommand{\MMapper}{\overline{\mathrm{M}}}
\newcommand{\im}{\mathrm{im}}
\newcommand{\Shift}{\mathrm{Shift}}
\newcommand{\Merge}{\mathrm{Merge}}
\newcommand{\Split}{\mathrm{Split}}
\newcommand{\Crit}{\mathrm{Crit}}
\newcommand{\Rips}{\mathrm{Rips}}
\newcommand{\disth}{d_{\rm H}}
\newcommand{\distfd}{d_{\rm FD}}
\newcommand{\distb}{d_{\rm B}}
\newcommand{\hdistfd}{\hat{d}_{\rm FD}}
\newcommand{\hdistb}{\hat{d}_{\rm B}}
\newcommand{\symdiff}{{\scriptstyle\triangle}}
\newcommand{\e}{\varepsilon}
\newcommand{\proba}[1]{\mathbb{P}\left(#1\right)}
\newcommand{\tf}{\tilde{f}}
\newcommand{\tg}{\tilde{g}}
\newcommand{\hei}{{\rm height}}
\newcommand{\ran}{{\rm range}}
\newcommand{\const}{22}
\newcommand{\RS}{{\sf Reeb}}
\newcommand{\accolade}[1]{\left\{\begin{array}{l}#1\end{array}\right.}

	\begin{abstract}
	Stable topological invariants are a cornerstone of persistence theory and applied topology, but their discriminative properties are often poorly-understood. In this paper we study a rich homology-based invariant first defined by Dey, Shi, and Wang in \cite{dey2015comparing}, which we think of as embedding a metric graph in the barcode space. We prove that this invariant is locally injective on the space of finite metric graphs and globally injective on a generic subset. 
	\end{abstract}
		\maketitle
	\thispagestyle{empty}
	
	\tableofcontents

	\clearpage
	\section{Introduction}
	
Much of applied geometry and topology is concerned with the extraction of computable and informative invariants from complex data sets. As is generally the case in data analysis, there is no ``perfect" or ``ideal" invariant. More discriminative invariants are generally more challenging to compute or compare, and may not be as stable as their simpler counterparts. A useful invariant will be one that balances these competing interests and thus finds a niche of successful application.\\
    
   For example, we can associate to a finite graph its diameter or average degree. As real numbers, both quantities are easy to store and compare. However, the average degree is much easier to approximate than the diameter. In terms of discriminativity, both fall very short from being injective, as there are widely different graphs with the same diameter or average degree. Lastly, both invariants are stable to certain deformations of the graph structure and unstable to others. With this understanding of our invariants, we have some sense of when they can, and should, be used.\\
    
     To give a more complex example, we may start with a compact metric space and compute the barcode of its \v Cech or Rips filtration. Depending on the size of the initial metric space (assuming it is finite) this process may be computationally intensive, and in addition the resulting barcode may not be much simpler than the metric space that produced it. However, the bottleneck metric between barcodes, unlike the Gromov-Hausdorff metric, is practically computable. On the stability side, the operator that maps a compact metric space to the barcode of its \v Cech or Rips filtration is Lipschitz continuous~\cite{chazal2014persistence}. However, on the discriminativity side, the operator performs poorly as, for instance, it is unable to discriminate between different geodesic trees---see e.g. Lemma~2 in~\cite{gasparovic2017complete} for the case of the \v Cech filtration, and Lemma~\ref{lem:Cech-Rips-hyper} in Appendix~\ref{sec:Rips-trees} for the case of the Rips filtration.\\

    The goal of this paper is to consider a rich persistence-based invariant, originally studied by Dey et al. in \cite{dey2015comparing}, that contains a persistence diagram for every point along a metric graph (not solely the vertices). This invariant can be used to define a pseudometric on the space of metric graphs via the Hausdorff metric on subsets of the space of persistence diagrams. The key results of \cite{dey2015comparing} are that this invariant is stable and approximable, and indeed Dey et. al. provide a polynomial-time algorithm to compute individual invariants and compare the Hausdorff distances between them. Our focus is on the remaining matter of discriminativity/injectivity: we show that the map that assigns a metric graph to its collection of barcodes is injective on a generic set of graphs.

		\subsection*{Prior Work} \label{prior work}
% 	In ``Comparing Graphs via Persistence Distortion" \cite{dey2015comparing}, Dey et al. propose the \emph{persistence distortion distance}, a pseudometric on the space of metric graphs given by immersing compact metric graphs in barcode space. They demonstrate the stability of their pseudometric relative to the Gromov-Hausdorff metric and analyze the features of the underlying immersion to develop fast polynomial-time computations. Our work complements theirs by considering the related inverse problem: when is the persistence distortion a bona fide metric and how can the original graph be recovered from its corresponding subset of barcode space?\\
	
	Beyond the setting of \cite{dey2015comparing} considered here, there have been few attempts so far at solving inverse problems in persistence theory. In \cite{turner2014persistent}, Turner et al. demonstrated that a simplical complex $X \subset \mathbb{R}^3$ can be recovered from a sphere's worth of filtrations. Whereas the approach of that paper is to employ extrinisically defined filtrations, our approach in this paper is intrinsic and local. In \cite{gameiro2016continuation}, Gameiro et al. developed an iterative technique for addressing inverse problems on point clouds. The techniques in \cite{gameiro2016continuation} show how a variation in barcode space can be used to produce a corresponding variation of  point clouds, but since there are multiple point clouds with the same barcode it is not possible to construct an inverse map from the space of barcodes to that of point clouds. The persistence distortion distance requires us to calculate many barcodes for the same space, and we show that this generically suffices for the existence of an inverse. Another related line of research is that of Gasparovic et al. in \cite{gasparovic2017complete}, where the 1-dimensional intrinsic \v{C}ech persistence diagram of a graph is related to the lengths of its shortest system of loops. This allows one to discern the minimal cycle lengths of a graph from a single persistence diagram. Finally, there is also the recent work of Curry in \cite{curry2017fiber} analyzing the fiber of the persistence map for functions on the interval.
	
\subsection*{Main Results}

The central contributions of this paper are a suite of inverse results for the barcode embedding first studied in \cite{dey2015comparing}:
\begin{itemize}
	\item Counterexample \ref{counterexample} demonstrates that the barcode embedding is not injective by presenting a pair of non-isometric trees with the same barcode embedding\footnote{\cite{dey2015comparing} also provides a counterexample to injectivity. Our example is quite different and particularly simple.}. By gluing small copies of these trees to other graphs, we can find counterexamples to injectivity in any Gromov-Hausdorff ball.
	\item Theorem \ref{btlocalinj} demonstrates that the barcode embedding is locally injective, in that the fiber of any given embedding (which is a subset of metric graphs) has no accumulation points in the Gromov-Hausdorff topology.
	\item Theorem \ref{psiinjbtinj} identifies a subset $S$ of the space of metric graphs on which the barcode embedding is globally injective. This subset is defined explicitly by a simple topological condition.
	\item Theorem \ref{densebamboo} asserts that this subset $S$ is dense in the space of metric graphs, equipped with the Gromov-Hausdorff topology.
	\item Theorem \ref{btgeninj} provides a genericity result for the subset $S$. It states that for most combinatorial graphs, almost every choice of edge lengths produces a metric graph in $S$. The combinatorial graphs excluded are precisely those for which no choice of edge lengths produces graphs in $S$.
	\item Theorem \ref{btgeninj2} extends Theorem \ref{btgeninj} by making no assumptions on combinatorial type. It states that excluding a measure-zero set of edge lengths for each combinatorial graph results in collection of metric graphs on which the barcode embedding is injective.  
\end{itemize}

In summary, although the barcode embedding is not an injective invariant for all metric graphs, this can be ameliorated by introducing a generic topological condition.

	\section{Background}\label{background}

    \subsection{Topological and Metric Graphs}

    \begin{definition}
    	A \emph{metric (topological) graph} is a metric (topological) space obtained by gluing together finitely many metric (topological) spaces isometric (homeomorphic) to line segments, such that only endpoints may be shared between segments. Note that a metric (topological) graph may contain self-loops or multiple edges between vertices.\\
    	
    	A topological graph can be turned into a metric graph by assigning a positive length to each of its constituent edges; up to isometry, all metric graphs can be obtained in this way. 
    \end{definition}

	\begin{definition}
		We will denote by $\mathbf{MGraphs}$ the space of all \emph{connected} metric graphs. Note that, by definition, $\mathbf{MGraphs}$ contains only graphs with finitely many edges and vertices.
	\end{definition}

    \begin{definition}
    \label{def:simplecycle}
    A \emph{simple cycle} in a graph $G$ is a sequence of vertices $v_{0}, v_{1}, \cdots, v_{n-1}, v_{n} = v_{0}$ with $[v_{0},v_{1}]$ an edge in $G$, and no vertex repeated aside from the \emph{base vertex} $v_{0} = v_{n}$.
    \end{definition}
    
    \begin{definition}
\label{def:selfloop}
A \emph{topological self-loop} in a metric graph $G$ is a simple cycle where every non-base vertex has valence equal to two. Any metric graph $G$ is a isometric to a graph $G'$ obtained from $G$ by deleting vertices of valence two and merging the adjacent edges. A topological self-loop in $G$ becomes a proper self-loop in the graph $G'$.
\end{definition}
    
    As a subspace of the space of compact metric spaces, $\mathbf{MGraphs}$ enjoys a natural metric, the Gromov-Hausdorff metric. This, in turn, induces a topology on $\mathbf{MGraphs}$ which we call the Gromov-Hausdorff topology.
    
    \begin{definition}
    \label{ghdef}
    Let $X,Y$ be any two compact metric spaces. A correspondence $\mathcal{M}$ between $X$ and $Y$ is a subset of $X \times Y$ whose projections to $X$ and $Y$ are surjective, i.e. $\pi_{X}(\mathcal{M}) = X$ and $\pi_{Y}(\mathcal{M}) = Y$. Intuitively, a correspondence pairs up elements of our two spaces so that each element of one space is paired up with at least one element of the other. The cost of a correspondence $\mathcal{M}$ is defined as
    \[\operatorname{cost}(\mathcal{M}) = \sup_{(x,y),(x',y') \in \mathcal{M}} |d_{X}(x,x') - d_{Y}(y,y')|\]
    
    The Gromov-Hausdorff distance between $X$ and $Y$ is then defined to be the infimum over the costs of all correspondences between them.
    \[d_{GH}(X,Y) = \inf_{\mathcal{M}} \operatorname{cost}(\mathcal{M})\]
    \end{definition}
    
    There are many other formulations of the Gromov-Hausdorff distance: interested readers can consult \cite{burago2001course}. An important result is that for compact spaces (as we are considering here), the infimum in the above definition is actually a minimum, i.e. it is realized by some (not necessarily unique) minimal cost correspondence, see \cite{ivanov2016realizations}.

%\subsection{Pointed Metric Graphs}

\begin{definition}
A pointed metric space is a pair $(X,p)$ where $X$ is a metric space and $p \in X$. If we take $X$ to be a metric graph, then we call $(X,p)$ a pointed metric graph,
\end{definition}

%A pointed metric space is a pair $(X,p)$ where $X$ is a metric space and $p \in X$. If we take $X$ to be a metric graph, we call $(X,p)$ a pointed metric graph, writing $\mathbf{PointedMGraphs}$ for the space of such objects. There is a pointed analogue of the Gromov-Hausdorff metric considered in \cite{jansen2017notes} that makes $\mathbf{PointedMGraphs}$ into a metric space, though we will not make use of it in this paper.

	\subsection{Metric-Measure Graphs}
\label{sec:mmg}
A metric-measure graph is a metric graph equipped with a Borel probability measure. For a general exposition on metric-measure spaces as a whole, see \cite{memoli2011gromov}. We will restrict our attention to metric-measure graphs of full support, i.e. pairs $(G,\mu_G)$ with $G = \supp{\mu_G}$, and denote by $\mathbf{MMGraphs}$ the space of all such objects. This space can be endowed with a variety of metrics, with one choice being the $ \mathfrak{D}_{\infty}$ metric defined in \cite{memoli2011gromov}, which we recall here. Given a pair of compact metric measure spaces $(X,d_X,\mu_X)$ and $(Y,d_Y,\mu_Y)$, a metric coupling $\pi$ is a measure on $X \times Y$ with $\mu_X$ and $\mu_Y$ as marginals. The set of all metric couplings is denoted $\mathcal{M}(\mu_X,\mu_Y)$. For a fixed metric coupling $\pi \in \mathcal{M}(\mu_X,\mu_Y)$ with support $\supp{\pi}$, we define the \emph{cost} of that coupling as
\[J_{\infty}(\pi) = \frac{1}{2} \sup_{(x,y),(x',y') \in \supp{\pi}} \Gamma_{X,Y}(x,y,x',y')\]
where 
\[\Gamma_{X,Y}(x,y,x',y') = |d_{X}(x,x') - d_{Y}(y,y')|\]

The $\mathfrak{D}_{\infty}$ metric is then defined to be the infimum of the cost $J_{\infty}(\pi)$ over all possible couplings $\pi$.
\[\mathfrak{D}_{\infty}((X,\mu_x),(Y,\mu_Y)) = \inf_{\pi \in \mathcal{M}(\mu_X,\mu_Y)} J_{\infty}(\pi)\]

The support of a measure coupling is always a correspondence (Lemma 2.2 of \cite{memoli2011gromov}), but not every correspondence between the supports $\supp{\mu_X}$ and $\supp{\mu_Y}$ of two measures comes from a measure coupling; thus, this quantity is generally larger than the Gromov-Hausdorff distance between supports.\\

\subsection{Morse-Type Functions}
\label{sec:Morse-type}

\begin{definition}\label{def:Morse-type}
A continuous real-valued function $f$ on a topological space $X$ 
is \emph{of Morse type} if:

(i) there is a finite set $\text{Crit}(f)=\{a_1<...<a_n\}\subset\R$, 
called the set of \emph{critical values},
such that over every open interval $(a_0=-\infty,a_1),...,(a_i,a_{i+1}),...,(a_n,a_{n+1}=+\infty)$
there is a compact and locally connected space $Y_i$  
and a homeomorphism $\mu_i:Y_i\times(a_i,a_{i+1})\rightarrow X^{(a_i,a_{i+1})}$
such that $\forall i=0,...,n, f|_{X^{(a_i,a_{i+1})}}=\pi_2\circ\mu_i^{-1}$, where
$\pi_2$ is the projection onto the second factor;

(ii) $\forall i=1,...,n-1,\mu_i$ extends to a continuous function $\bar{\mu}_i:Y_i\times[a_i,a_{i+1}]\rightarrow X^{[a_i,a_{i+1}]}$;
 similarly, $\mu_0$ extends to $\bar{\mu}_0:Y_0\times(-\infty,a_1]\rightarrow X^{(-\infty,a_1]}$
and $\mu_n$ extends to $\bar{\mu}_n:Y_n\times[a_n,+\infty)\rightarrow X^{[a_n,+\infty)}$;

(iii) Each levelset $f^{-1}(t)$ has a finitely-generated homology.
\end{definition}

Let us point out that a Morse function is also of Morse type, and that
its critical values remain critical in the definition above.  Note
that some of its regular values may be termed critical as well in this
terminology, with no effect on the analysis.\\

\subsection{Extended Persistent Homology}

Extended persistence is a modification of sublevel set persistence that combines the information contained in the sublevel set and superlevel set filtrations of a height function on a topological space. This construction was defined in great generality by Cohen-Steiner et. al. \cite{cohen2009extending}; for the sake of completeness, we provide a concise definition below.\\

Let $f$ be a real-valued function on a topological space $X$.  The
family $\{X^{(-\infty, \alpha]}\}_{\alpha\in\R}$ of sublevel sets of
  $f$ defines a {\em filtration}, that is, it is nested
  w.r.t. inclusion: $X^{(-\infty, \alpha]}\subseteq X^{(-\infty,
  \beta]}$ for all $\alpha\leq\beta\in\R$. The family $\{X^{[\alpha,
      +\infty)}\}_{\alpha\in\R}$ of superlevel sets of $f$ is also
    nested but in the opposite direction: $X^{[\alpha,
        +\infty)}\supseteq X^{[\beta, +\infty)}$ for all
        $\alpha\leq\beta\in\R$. We can turn it into a filtration by
        reversing the order on the real line. Specifically, let
        $\Rop=\{\tilde{x}\ |\ x\in\R\}$, ordered by
        $\tilde{x}\leq\tilde{y}\Leftrightarrow x\geq y$. We index the
        family of superlevel sets by $\Rop$, so now we have a
        filtration: $\{X^{[\tilde\alpha,
            +\infty)}\}_{\tilde\alpha\in\Rop}$, with
          $X^{[\tilde\alpha, +\infty)}\subseteq X^{[\tilde\beta,
                +\infty)}$ for all
              $\tilde\alpha\leq\tilde\beta\in\Rop$.

Extended persistence connects the two filtrations at infinity as
follows. First, replace each superlevel set $X^{[\tilde\alpha, +\infty)}$ by
  the pair of spaces $(X, X^{[\tilde\alpha, +\infty)})$ in the second
    filtration. This maintains the filtration property since we have
    $(X, X^{[\tilde\alpha, +\infty)})\subseteq (X, X^{[\tilde\beta,
          +\infty)})$ for all
        $\tilde\alpha\leq\tilde\beta\in\Rop$. Then, let
        $\Rext=\R\cup\{+\infty\}\cup\Rop$, where the order is
        completed by $\alpha<+\infty<\tilde{\beta}$ for all
        $\alpha\in\R$ and $\tilde\beta\in\Rop$. This poset is
        isomorphic to $(\R, \leq)$. Finally, define the {\em extended
          filtration} of $f$ over $\Rext$ by:
\[
F_\alpha = X^{(-\infty, \alpha]}\ \mbox{for $\alpha\in\R$},\ F_{+\infty} = X \equiv (X,\emptyset)\text{ and }
F_{\tilde\alpha} = (X, X^{[\tilde\alpha, +\infty)})\ \mbox{for $\tilde\alpha\in\Rop$},
\]
where we have identified the space $X$ with the pair of spaces $(X,
\emptyset)$ at infinity. 
The subfamily $\{F_\alpha\}_{\alpha\in\R}$ is
  the \textit{ordinary} part of the filtration, while $\{F_{\tilde\alpha}\}_{\tilde\alpha\in\Rop}$ is the
  \textit{relative} part. 

Applying the homology functor $H_*$ to this filtration gives the so-called \textit{extended persistence module} $\mathbb{V}$ of $f$,
which is a family of vector spaces connected by linear maps induced by the
inclusions in the extended filtration.
For functions having finitely many critical values (i.e. those of \emph{Morse type}), like the distance from a fixed basepoint in a metric graph, the extended persistence module can be
decomposed as a finite direct sum of half-open {\em interval
  modules}---see e.g.~\cite{chazal2012structure}:
$\mathbb{V}\simeq\bigoplus_{k=1}^n \mathbb{I}[b_k, d_k)$,
where each summand $\mathbb{I}[b_k, d_k)$ is made of copies of the
  field of coefficients at every index $\alpha\in [b_k, d_k)$, and of
    copies of the zero space elsewhere, the maps between copies of the
    field being identities.  Each summand represents the lifespan of a
    {\em homological feature} (connected component, hole, void, etc.) within the
    filtration. More precisely, the {\em birth time} $b_k$ and {\em
      death time} $d_k$ of the feature are given by the endpoints of
    the interval.  Then, a convenient way to represent the structure
    of the module is to plot each interval in the decomposition as a
    point in the extended plane, whose coordinates are given by the
    endpoints. Such a plot is called the \textit{extended persistence
      diagram} of $f$, denoted $\Dg(f)$.  The distinction
    between ordinary and relative parts of the filtration allows us to
    classify the points in $\Dg(f)$ as follows:
\begin{itemize}
\item $p=(x,y)$ is called an {\em ordinary} point if $x,y\in\R$; 
\item $p=(x,y)$ is called a {\em relative} point if $\tilde{x},\tilde{y}\in\Rop$;
\item $p=(x,y)$ is called an {\em extended} point if $x\in\R,\tilde{y}\in\Rop$; 
\end{itemize}
Note that ordinary points lie strictly above the diagonal
$\Delta=\{(x,x)\ |\ x\in\R\}$ and relative points lie strictly below $\Delta$,
while extended points can be located anywhere, including on $\Delta$ (e.g. when a connected component 
lies inside a single critical level). 
It is common to
partition $\Dg(f)$ according to this classification:
$\Dg(f)=\Ord(f)\sqcup\Rel(f)\sqcup\Ext^+(f)\sqcup\Ext^-(f)$, where $Ext^{+}(f)$ are those extended points above the diagonal, $Ext^{-}(f)$ are those below the diagonal, and by
convention $\Ext^+(f)$ includes the extended points located on the
diagonal~$\Delta$.\\

Persistence diagrams are also commonly referred to as \emph{barcodes} in the literature, and we will use the notation {\bf Barcodes} to reference the space of these objects. This space enjoys a natural pseudometric: the \emph{bottleneck distance}, defined as follows. For two diagrams $D_1 = \{p_1, \cdots, p_n\}$ and $D_2 = \{q_1, \cdots, q_k\}$, one defines a matching $\mathcal{M}$ between $D_1$ and $D_2$ to be a pairing between a subset of points in $D_1$ and a subset of equal size of points in $D_2$, where all the points not in these subsets are thought of as being paired with the diagonal $\Delta$. The cost a pairing between points $p_i$ and $q_j$ is the $L^\infty$ distance between them. To define the cost of pairing a point with the diagonal, we first define $\pi_{\Delta}(p)$ to be the closest point to $p$ on the diagonal $\Delta$ in the $L^\infty$ metric. Then the cost of the pairing $(p,\Delta)$ is the $L^\infty$ distance between $p$ and $\pi_{\Delta}(p)$. The cost of a matching $\mathcal{M}$ is then the maximum cost of any of its pairings, and the bottleneck distance between $D_1$ and $D_2$ is the infimum of the cost of matchings between them. This is a pseudometric if we allow our diagrams to have points on the diagonal, but otherwise is a proper metric. { Note that $d_{B}$ is oblivious to the labels $\Ext, \Ord$, and $\Rel$, but not to the dimension. Likewise, in this paper, we assume that our persistence diagrams come labelled by dimension, but do not have any labels indicating if a particular point comes from ordinary, relative, or extended persistence.} For more on the theoretical and computational aspects of the bottleneck distance, see \cite{kerber2017geometry}.\\

In the setting of this paper, the utility of using extended persistence in place or ordinary persistence is twofold. Firstly, it will be crucial for stability, as adding a small loop to a graph produces a new graph which is similar to the original one in the Gromov-Hausdorff sense but will not have a similar ordinary persistence diagram, as it now contains a new feature that lives forever, and hence sits infinitely far away from the diagonal; not so in extended persistence, where that feature dies shortly before it is born. Secondly, for $G$ contractible (i.e. a tree), there is no interesting ordinary homology, but the identifications coming from the relative part of the filtration produce many features whose birth times and death times correspond to when boundary leaves appear and when branches merge.\\

We have the following general stability result:

\begin{theorem}[\cite{chazal2012structure}, \S 6.2]
    \label{perstable}
    Let $X$ be a topological space homeomorphic to the realization of a simplical complex, and $f,g : X \to \mathbb{R}$ two continuous functions whose sublevelset and superlevelset filtrations give extended persistence barcodes $B_f$ and $B_g$. Then $d_{B}(B_f,B_g) \leq \norm{f-g}_{\infty}$.
\end{theorem}

Since every topological graph can be turned into a simplical complex by adding extra vertices along self-loops and multiple edges, and since all the functions we will be considering are continuous, this result applies in our setting.

	\subsection{Reeb Graphs}
	\label{reebbackground}
	
	A construction that plays an important role in our analysis is that of a Reeb graph. We recall the definition here.
	
	\begin{definition}
	Given a topological space $X$ and a continuous function $f:X \to \mathbb{R}$, one can define an equivalence relation $\sim_{f}$ between points in $X$, where $x \sim_{f} y$ if only if $f(x) = f(y)$ and $x,y$ belong to the same  path-component of $f^{-1}(f(x)) = f^{-1}(f(y))$. The \emph{Reeb graph} $R_{f}(X)$ is the quotient space $X/\sim_f$, and since $f$ is constant on equivalence classes, it descends to a well-defined function on $R_{f}(X)$.\\
	
	The space $R_{f}(X)$ is equipped with the following metric.
	\begin{align*}d_{f}(x,x') = \min_{\pi: x \to x'} \left\{ \max_{t \in [0,1]} f \circ \pi(t) - \min_{t \in [0,1]} f \circ \pi(t)\right\} \\ \mbox{ where $\pi: [0,1] \to R_f$ ranges over all continuous paths from $x$ to $x'$ in $R_f$.}
	\end{align*}
	\end{definition}

	Let us denote by  $\mathbf{Reeb}$ the space of all compact Reeb graphs. $\mathbf{Reeb}$ admits a few natural metrics, with one common choice being the \emph{functional distortion distance}, or FD distance. It is defined as follows
	
	\begin{definition}[\cite{reebgraphdistance}, \S 3]
	\label{FDdef}
	\[d_{FD}(R_f,R_g) = \inf_{\phi,\psi} \max \left\{ \norm{f - g\circ \phi}_{\infty}, \norm{f \circ \psi - g}_{\infty}, \frac{1}{2}D(\phi,\psi) \right\}\]
	
	where
	\begin{itemize}
	\item $\phi: R_f \to R_g$ and $\psi: R_g \to R_f$ are continuous maps.
	\item $D(\phi,\psi) = \sup \left\{ |d_{f}(x,x') - d_{g}(y,y')| \mbox{ such that } (x,y),(x',y') \in C(\phi,\psi) \right\}$ where $C(\phi,\psi) = \left\{ (x,\phi(x))| x \in R_f \right\} \cup \left\{ (\psi(y),y) | y \in R_{g}\right\}$
	\end{itemize}
	
	\end{definition}
	
	Thus, the FD distance has three components, two that make sure that the approximating maps preserve the $f$- and $g$-functions, and a third which ensures that distances in $R_f$ and $R_g$ are preserved through the maps $\phi, \psi$. The following lemma states that the FD distance is well-defined on isometry classes of Reeb graphs.
	
	\begin{lemma}
	\label{FDzero}
	Let $R_f$ and $R_g$ be Reeb graphs. Suppose that $d_{FD}(R_f,R_g) = 0$. Then there are a pair of isometries $\phi: R_f \to R_g$ and $\psi: R_g \to R_f$ which preserve functions, i.e. $\forall x \in R_f$, $f(x) = g(\phi(x))$, and $\forall y \in R_g$, $g(y) = f(\psi(y))$.
	\end{lemma}
	
	\begin{proof}
	Let $\phi_{n}$ and $\psi_{n}$ be continuous maps for which the term 
	\[ \frac{1}{2}D(\phi_n,\psi_n), \norm{f - g\circ \phi_n}_{\infty}, \norm{f \circ \psi_n - g}_{\infty} \]
	is less than $1/n$, i.e. approximate matchings between our Reeb graphs whose \emph{distortion} approaches the infinimum, which in this case is zero. The fact that the term $\frac{1}{2}D(\phi_n,\psi_n)$ goes to zero means that the Gromov-Hausdorff distance between $R_f$ and $R_g$ goes to zero, and hence, by Theorem 7.3.30 in \cite{burago2001course}, $R_f$ and $R_g$ are isometric. In fact, the proof of that theorem demonstrates that some subsequence of these maps $\phi_n$ and $\psi_n$ converges to a pair of isometries $\phi,\psi$. The requirements that $\norm{f - g\circ \phi_n}_{\infty} \to 0$ and $\norm{f \circ \psi_n - g}_{\infty} \to 0$ ensure that these isometries preserve the height functions.
\end{proof}

The height function $f$ on a Reeb graph $R_f$ gives rise to sublevel and superlevel set filtrations, from which we can compute an extended persistence diagram. One can then compare the FD distance between two Reeb graphs and the bottleneck distances between their 1-dimensional extended persistence barcodes. If our Reeb graphs are of \emph{Morse type}, then we have the following result from \cite{localequiv}

	\begin{theorem}[\cite{localequiv}]
	\label{ReebInequality}
	Let $R_f$ and $R_g$ be two Reeb graphs, with critical values $\{a_1, \cdots , a_n\}$ and $\{b_1, \cdots, b_m\}$. Let $a_{f} = \min_{i} \{a_{i+1} - a_{i}\}$, and $a_{g} = \min_{i} \{b_{i+1} - b_{i}\}$. Let $K \in (0,1/22]$. If $d_{FD}(R_f,R_g) \leq \max \{a_f,a_g\}/(8(1+22K))$, then
	\[Kd_{FD}(R_f,R_g) \leq d_{B}(R_f,R_g) \leq 2d_{FD}(R_f,R_g)\]
	
	where $d_{B}(R_f,R_g)$ is the bottleneck distance between their respective extended persistence barcodes. Note that the upper bound on the bottleneck distance holds more generally for any pair of Reeb graphs; the qualifications are necessary only for the lower bound.
\end{theorem}

Incorporated in this theorem are the two facts that taking extended persistence barcodes is stable and locally injective (and in fact locally bilipschitz). In this context, local injectivity means that a fixed Reeb graph $R_f$ has a distinct barcode from any nearby Reeb graphs; however, there may be a pair of Reeb graphs near $R_f$ with the same barcode. Put another way, the operation of assigning a barcode to a Reeb graph is not fully injective on the small open ball described in the theorem -- comparisons are only allowed with the fixed reference Reeb graph at the center of the ball. In fact, generically speaking, there exist arbitrarily close pairs of Reeb graphs with the same persistence diagram.

\subsubsection{Dictionary}
	\label{sec:reebdictionary}
An explicit dictionary between geometric features in a Reeb graph and the points in its extended persistence diagram is given in \cite{reebgraphdistance}. We recall it here.\\

 A \emph{downfork} is a point $v \in X$ together with a pair of adjacent directions along which $f$ is locally decreasing (if $v$ is a leaf vertex one direction suffices), and an \emph{upfork} is a point $u \in X$ together with a pair of adjacent directions along which $f$ is locally increasing (if $v$ is a leaf vertex, one direction suffices here too). We distinguish between two types of downforks: \emph{ordinary} downforks where the two adjacent directions sit in different connected components of the sublevelset $f^{-1}((-\infty, f(v)])$, and \emph{essential} downforks where they are in the same connected component. There is a similar dichotomy for upforks, giving four types of forks in total.\\
 
 We now identify the eight kinds of endpoints appearing in the extended persistence diagram of a Reeb graph. In $H_{0}$ there are only two kinds of points: ordinary and extended, as new connected components do not appear in the relative part of the filtration. In $H_{1}$ there are only relative and extended points, as the lack of $2$-dimensional faces in our Reeb graph means that no cycle dies in ordinary homology. Each such point has two endpoints, a birth and a death time, giving a total of eight types of endpoints.\\ 

Thus, we have the dictionary prescribed in Table \ref{tab:dict}. Note that this dictionary does not promise a canonical or unique assignment of a downfork or upfork to an endpoint in the persistence diagram, nor vice versa. Rather, it is meant to assert that the existence of certain endpoints in the persistence diagram guarantees the existence of certain downforks or upforks. See Figure \ref{fig:persexample} for an illustration of this pairing for a particular graph.  Readers can find more details in \cite{reebgraphdistance}.

\begin{center}
	\begin{table}
\begin{tabular}{ | c || c |}
	\hline
	\multicolumn{2}{|c|}{Dictionary for Extended Persistence of Reeb Graphs} \\
	\hline
	Persistence Diagram    & Reeb Graph\\
	\hline
 $(f(v),\cdot)$ in ordinary $H_0$ persistence. & upfork $v$ (ordinary or essential).\\
 $(\cdot,f(v))$ in ordinary $H_0$ persistence. & ordinary downfork $v$.\\
 $(f(v),\cdot)$ in extended $H_0$ persistence. & upfork $v$ (ordinary or essential).\\
 $(\cdot,f(v))$ in extended $H_0$ persistence. & downfork $v$ (ordinary or essential).\\
 $(f(v), \cdot)$ in relative $H_1$ persistence. &   downfork $v$ (ordinary or essential).\\
$(\cdot,f(v))$ in relative $H_1$ persistence. &  ordinary upfork $v$.\\
$(f(v),\cdot)$ in extended $H_1$ persistence. &  essential downfork $v$.\\
$(\cdot,f(v))$ in extended $H_1$ persistence. & essential upfork $v$.\\
	\hline

\end{tabular}
\caption{A dictionary between endpoints in the extended persistence diagram of a Reeb graph and the various kinds of upforks and downforks it contains.}
\label{tab:dict}
\end{table}
\end{center}

 \begin{figure}[htb]
 	\includegraphics[scale=0.9]{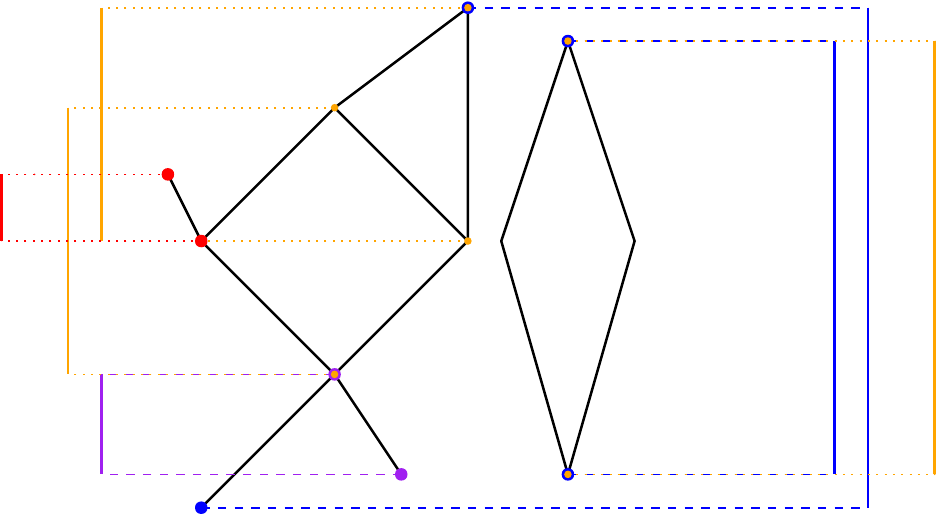}
 	\caption{Consider the above Reeb graph, whose associated function is simply the projection on to the vertical axis. This Reeb graph is disconnected, having two connected components. In $H_0$, ordinary diagram points (shown here as intervals) are in purple and extended points are in blue. In $H_1$, relative points are in red and extended points are in orange. In $H_0$, there is one ordinary point, born at an ordinary upfork and dying at an ordinary downfork. There are two extended points in $H_0$, one born at an ordinary upfork and the other at an essential upfork, and both dying at essential downforks. In $H_1$, there is one relative point, born at an ordinary downfork and dying at an ordinary upfork. There are three extended points in $H_1$, all born at essential downforks and dying at essential upforks. Note that, for both relative and extended points in $H_1$, the death occurs at a lower value than the birth.}
 	\label{fig:persexample}
 \end{figure}

\begin{lemma}
\label{lem:upforkvertex}
Given a connected metric graph $G$ and a point $p \in G$, consider the Reeb graph $R_f$ associated to the function $f(\cdot) = d_{G}(p,\cdot)$. Any upfork in $R_f$ distinct from $p$ is necessarily a vertex of valence at least three. Taken together with the prior dictionary, this implies that the set of nonzero death times in the one-dimensional persistence diagram of $R_f$ is a subset of the set of distances from $p$ to vertices of valence at least three. 
\end{lemma}
\begin{proof}
Let $x \neq p$ be a point in $G$ which is not a vertex of valence at least three. Then there are at most two adjacent directions to $x$ in $G$. It is impossible for such a point to be an upfork, as at least one of the directions adjacent to it is the initial segment of a geodesic from $x$ to $p$. If $x$ is a leaf, there is only one such adjacent direction, and so $x$ is a downfork. If $x$ has valence two, then there is at most one direction along which the distance to $p$ (and hence the value of $f$) can be locally increasing, and again $x$ cannot be an upfork.  
\end{proof}

  \begin{remark}
  
    Consider again the setting of Lemma \ref{lem:upforkvertex}, in which the Reeb graph is constructed from a connected metric graph $G$ using a function of the form $d_{G}(p,\cdot)$. As the sublevel sets of $d_{G}(p,\cdot)$ are always connected (the distances in $G$ can be realized by geodesics), there is only one point in $H_0$ persistence; this point appears in extended persistence, and is born at time $0 = d_{G}(p,p)$, corresponding to the upfork $p$. Note also that although nonzero one-dimensional death times correspond to distances to vertices of valence at least three, a birth time need not correspond to the distance to any vertex. Consider for example a graph $G$ with three vertices $v_1,v_2,v_3$ at pairwise distance $1$, forming an equilateral triangle. The only point in the one-dimensional persistence diagram $\Psi_{G}(v_1)$ corresponds to the loop generated by the full graph, and is born at distance $1.5$, the radius of $G$ at $v_1$. In the Reeb graph $\Phi_{G}(p)$, this is the distance from $p$ to the downfork point halfway between $v_2$ and $v_3$. 
    \end{remark}

\section{The Barcode Embedding and the Persistence Distortion Pseudometric}
\label{constructionsection}

The goal of this section is to review and expound upon the relevant constructions and results in \cite{dey2015comparing} from a metric geometry perspective, incorporating some additional results for metric graphs and metric measure graphs.\\  

The following diagram introduces the spaces and maps involved in our constructions; $G$ is any fixed metric graph.
\begin{center}
\begin{tikzcd}
{\bf PointedMGraphs} \arrow{r}{\Phi} \arrow[bend left]{rr}{\Psi} & {\bf Reeb} \arrow{r}{\operatorname{ExDg}} & {\bf Barcodes}\\
G \arrow{ur}{\Phi_G} \arrow[bend right]{rru}{\Psi_G} & &
\end{tikzcd}
\end{center}
\vspace{10 mm}

	\begin{definition}
	Let $(G,x)$ be a pointed metric graph. We will write $\Phi(G,x)$ to denote the Reeb graph given by the function $f_{x} : G \to \mathbb{R}$ defined by $f_{x}(y) = d_{G}(x,y)$. If we fix a metric graph $G$ then we can view $\Phi$ as a function solely of this basepoint, writing $\Phi_{G}(x) = \Phi(G,x)$.
	\end{definition}
	
The maps $\Phi$ and $\Phi_G$ are identical in nature but differ in their domain: the former is defined on the entire space of pointed graphs, whereas the latter is defined on a fixed graph $G$. Now, given a Reeb graph we can compute its extended persistence diagram, and we denote this operation by $\operatorname{ExDg}$. Composing $\Phi$ and $\Phi_G$ with $\operatorname{ExDg}$ produces our second pair of maps.

	\begin{definition}
	Let $(G,x)$ be a pointed metric graph. Define $\Psi(G,x) = \operatorname{ExDg} \circ \Phi(G,x)$ to be the extended persistence diagram of the Reeb graph associated to $(G,x)$. Similarly, for a fixed metric graph $G$, we define $\Psi_{G}(x) = \operatorname{ExDg} \circ \Phi_{G}(x)$.
	\end{definition}

 The following results demonstrate that $\Phi_{G}$ and $\Psi_{G}$ are Lipschitz and that similar graphs produce similar Reeb graphs and barcodes.

\begin{lemma}
    \label{psilip}
    Fix a metric graph $G$. Let $p,q \in G$ be any two basepoints. Then $d_{FD}(\Phi_{G}(p),\Phi_{G}(q)) \leq d_{G}(p,q)$ and $d_{B}(\Psi_{G}(p),\Psi_{G}(q)) \leq d_{G}(p,q)$.
\end{lemma}
\begin{proof}
    The $L^{\infty}$ distance between the two distance functions $d(p, \cdot)$ and $d(q, \cdot)$ is bounded by $d_{G}(p,q)$ by the triangle inequality. Thus the first claim can be seen by considering by setting $\phi$ and $\psi$ equal to the identity map in the definition of the functional distortion distance, and the second follows from Theorem \ref{perstable}.
\end{proof}

\begin{theorem}[\cite{dey2015comparing}, \S 3]
\label{pairedgraphs}
    Let $G,G'$ be a pair of metric graphs, and let $\mathcal{M}$ be a correspondence between them realizing the Gromov-Hausdorff distance $\delta = d_{GH}(G,G')$. If $p \in G$ and $p' \in G'$ are a pair of points with $(p,p') \in \mathcal{M}$ then the two Reeb graphs $\Phi_{G}(p)$ and $\Phi_{G'}(p')$ are within $6\delta$ of each other in the functional distortion distance, and the resulting barcodes $\Psi_{G}(p)$ and $\Psi_{G'}(p')$ are within $18\delta$ of each other in the bottleneck distance.
\end{theorem}

The next step in our construction is to consider the collection of Reeb graphs produced by varying the basepoint along a fixed metric graph. The set of extended persistence barcodes for these Reeb graphs is our intended object of study. For a topological space $X$, the notation $C(X)$ refers to the set of compact subsets of $X$.

\begin{definition}
	For a fixed metric graph $G$, define the \emph{Reeb transform} of $G$ to be the collection of Reeb graphs $\mathcal{RT}(G) = \{\Phi_{G}(x) \mid x \in G \}$. The resulting set of barcodes $\mathcal{BT}(G) = \{\Psi_{G}(x) \mid x \in G\}$ will be called the \emph{barcode transform} of $G$.
	
	\begin{center}
		\begin{tikzcd}
		\mathbf{MGraphs} \arrow{r}{\mathcal{RT}} \arrow[bend right]{rr}{\mathcal{BT}} & C(\mathbf{Reeb}) \arrow{r}{\operatorname{ExDg}} & C(\mathbf{Barcodes})
		\end{tikzcd}
	\end{center}
\end{definition}

The following lemma demonstrates that the barcode transform of a metric graph has a natural metric structure.

\begin{lemma}
    \label{bottleneckmetric}
    Let $G$ be a metric graph. The bottleneck distance $d_{B}$ restricted to $\mathcal{BT}(G)$ is always a true metric.
\end{lemma}
\begin{proof}
    We distinguish two cases: either $G$ is a graph consisting of a single point or it is not. If $G$ is a single point, $\mathcal{BT}(G)$ is a single barcode consisting of the point $(0,0)$ on the diagonal, and $d_B$ restricts to give the trivial metric on this space. Otherwise, if $G$ not the one-point graph, the dictionary of section \ref{sec:reebdictionary} implies that none of the Reeb graphs it produces have barcodes with points on the diagonal, and as mentioned in Section \ref{reebbackground} the bottleneck distance is a metric if our barcodes avoid the diagonal. 
\end{proof}

If one starts instead with a metric measure graph $(G,\mu)$, then the maps $\Phi_{G}$ and $\Psi_{G}$ can be used to push forward the measure on $G$ to measures on the spaces $\mathbf{Reeb}$ and $\mathbf{Barcodes}$ respectively. For a measurable space $Y$, the notation $\mathcal{P}(Y)$ refers to the space of Borel probability measures on $Y$.

\begin{definition}
For a fixed metric measure graph $(G,\mu)$, define the pushforward measures $\mathcal{RMT}(G,\mu) = (\Phi_{G})_{\ast}(\mu) \in  \mathcal{P}(\mathbf{Reeb})$ and $\mathcal{BMT}(G,\mu) = (\Psi_{G})_{\ast}(\mu) \in \mathcal{P}(\mathbf{Barcodes})$ as the \emph{Reeb measure transform} and \emph{barcode measure transform} respectively.
\begin{center}
\begin{tikzcd}
\mathbf{MMGraphs} \arrow{r}{\mathcal{RMT}} \arrow[bend left]{rr}{\mathcal{BMT}} & \mathcal{P}(\mathbf{Reeb}) \arrow{r}{\operatorname{ExDg}}  & \mathcal{P}(\mathbf{Barcodes})
\end{tikzcd}
\end{center}
\end{definition}

In the following, our focus will mainly be on the maps $\mathcal{BT}$ and $\mathcal{BMT}$; the maps $\mathcal{RT}$ and $\mathcal{RMT}$ are defined for the sake of completeness but are of little independent interest. Indeed, the following two results demonstrate that a pointed metric graph can be recovered from any of its associated Reeb graphs, and that this inverse map is well-defined on the space of equivalence classes of Reeb graphs. The proofs of these results can be found in Section \ref{reebproofs}.

	\begin{lemma}
	\label{recovergraph}
	Given $(G,x)$, a pointed, connected, and compact metric graph, let $R_{f} = \Phi(G,x)$ be the associated Reeb graph. Define a metric on $R_{f}$ using the infinitesimal length element $|df|$. That is, if $y$ and $z$ are two points on $R_f$, and $\Gamma$ is the set of continuous paths from $y$ to $z$, then
	\[d_{f}(y,z) := \inf_{\gamma \in \Gamma} \int_{\gamma} |df|.\]
	
	Then, $R_{f}$, equipped with the metric $d_{f}$, is isometric to the original metric graph $(G,d_{G})$.
\end{lemma}

	\begin{corollary}
	\label{pointedinj}
	Let $(G,x)$ and $(G',x')$ be two pointed metric graphs. If $\Phi(G,x) \simeq \Phi(G',x')$ then there is an isometry  $f: G \to G'$ with $f(x) =x'$.
\end{corollary} 

% \begin{remark}
% It is worth noting why we have distinguished between the similar maps $\Psi$ and $\Psi_{G}$. The map $\Psi$ has the graph $G$ as an argument and will be considered when trying to compare distinct graphs for the purposes of proving stability. However, there is much more that can be said about $\Psi$ when the graph is fixed and only the basepoint is moved, and we will need to consider this case carefully for the geometric reconstruction results in our section on discriminativity. The notation $\Psi_{G}$ emphasizes this latter type of variation and makes it clear that one should be visualizing a fixed graph.
% \end{remark}

The barcode transform and barcode measure transform can be used to define pseudometrics on the space of metric graphs and metric measure graphs respectively, as in \cite{dey2015comparing}.

\begin{definition}
\label{persistencedistortion}
Let $G,H$ be a pair of metric graphs. Define the \emph{persistence distortion} pseudometric $d_{PD}(G,H) = d_{H}(\mathcal{BT}(G),\mathcal{BT}(H))$ to be the Hausdorff distance between their corresponding subsets of Barcode space. Similarly, if $(G,\mu_G)$ and $(H,\mu_H)$ are a pair of metric measure graphs, we define the \emph{measured persistence distortion} pseudometric using the $\infty$-Wasserstein metric, $d_{MPD}((G,\mu_G),(H,\mu_H)) = {d}_{W,\infty}(\mathcal{BMT}(G,\mu_G),\mathcal{BMT}(H,\mu_H))$.
\end{definition}

\section{Stability}

The following stability theorem demonstrates that the persistence distortion can provide a weak lower bound to the intractable Gromov-Hausdorff distance, and is Theorem 3 in \cite{dey2015comparing}, with the constant changed from $6$ to $18$ as per the comment following Theorem 4, since we are using 1-dimensional persistence in addition to 0-dimensional.

\begin{theorem}[\cite{dey2015comparing}]
\label{pdstable}
For a pair of metric graphs $G,H$, $d_{PD}(G,H) \leq 18 \, d_{GH}(G,H)$.    
\end{theorem}

An identical result holds for the measured persistence distortion.

\begin{theorem}
\label{mpdstable}
For a pair of (full-support) metric measure graphs $(G,\mu_G)$ and $(H,\mu_H)$,
\[d_{MPD}((G,\mu_G),(H,\mu_H)) \leq 18 \, \frak{D}_{\infty}((G,\mu_G),(H,\mu_H))\] 
\end{theorem}

The rest of this section is devoted to the proof of Theorem \ref{mpdstable}. 

\begin{proof}
    The strategy of this proof is to show that a measure coupling between a pair of metric-measure graphs gives rise to a low-distortion correspondence between their supports, which pushes forward to a low-distortion correspondence between subsets of barcode space. We then show that this correspondence between subsets of barcode space is dense in the support of the pushforward of the measure coupling, so that the pushforward measures admit a coupling with support very close to the diagonal.  \\ 

	Let $G,H \in \mathbf{MMGraphs}$ such that $\mathfrak{D}_{\infty}(G,H) = \delta$. For a given $\epsilon > 0$, let $\pi$ be a measure coupling of $\mu_{G}$ and $\mu_{H}$ for which $J_{\infty}(\pi) < \delta + \epsilon$. As we have seen in Section \ref{sec:mmg}, the support of a measure coupling is always a correspondence, and since $\mu_G$ and $\mu_H$ have full measure, we know that $\supp{\pi}$ is a correspondence between $G = \supp{\mu_G}$ and $H = \supp{\mu_H}$. Moreover, since $d_{GH}(G,H) = \frac{1}{2}\inf_{R} \norm{\Gamma_{G,H}}_{L^{\infty}(R \times R)}$, we can deduce that $d_{GH}(G,H) < \delta + \epsilon$. Using the proof of Theorem \ref{pdstable} from \cite{dey2015comparing}, if $(x,y) \in \supp{\pi}$ are pairs of points paired by the correspondence then $d_{B}(\Psi(G,x),\Psi(H,y)) \leq 18(\delta + \epsilon)$. Now consider the pushforward measures $\tilde{\mu}_G = (\Psi_{G})_{\ast}(\mu_G)$, $\tilde{\mu}_H = (\Psi_{H})_{\ast}(\mu_H)$, and $\tilde{\pi} = (\Psi_{G} \times \Psi_{H})_{\ast}(\pi) \in P(\mathbf{Barcodes} \times \mathbf{Barcodes})$.\\
	
	The following claim, which we prove later on in this section, allows us to make use of $\tilde{\pi}$ in estimating $\frak{D}_{\infty}$. 
	
	\begin{claim*}
	$\tilde{\pi}$ is a measure coupling of $\tilde{\mu}_G$ and $\tilde{\mu}_H$.
	\end{claim*}
	
	Next, we claim that the image of the support of $\pi$ is dense in the support of $\tilde{\pi}$. This is a general fact about measures and pushforwards, which we likewise prove later on in this section.
	
	\begin{claim*}
	Let $P$ be a Polish space, and $T$ any topological space, and let $f: P \to T$ be continuous. If $\Pi$ is a probability measure on $P$, then $f(\supp{\pi}))$ is dense in $\supp{f_{\ast}(\pi)}$.
	\end{claim*}
	
	Now take a pair $(b_1,b_2) \in \supp{\tilde{\pi}}$. By density, there is an arbitrarily close pair $(b_1',b_2') \in (\Psi_{G} \times \Psi_{H})(\supp{\pi})$, corresponding to a pair $(x,y) \in \supp{\pi}$ for which $b_1'$ is the barcode associated to the pointed graph $(G,x)$, and similarly for $b_2'$ and $(H,y)$. As we have shown, the distance between the barcodes $b_1'$ and $b_2'$ is at most $18(\delta + \epsilon)$. By the triangle inequality, and taking limits, the same is true of the pair $(b_1,b_2)$. Finally, letting $\epsilon$ go to zero completes the proof.    
\end{proof}

We now prove the two claims used in the proof of Theorem \ref{mpdstable}.

\begin{claim}
	$\tilde{\pi}$ is a measure coupling of $\tilde{\mu}_G$ and $\tilde{\mu}_H$.
\end{claim}
\begin{proof}
	Let us show that $\tilde{\pi}$ has these measures as marginals. Let $S \subset \mathbf{Barcodes}$ be a measurable subset of the barcode space
	\begin{align*}
	\tilde{\pi}(S \times \mathbf{Barcodes}) & = \pi((\Psi_{G} \times \Psi_{H})^{-1}(S \times \mathbf{Barcodes}))\\
	& = \pi(\psi_{G}^{-1}(S) \times H)\\
	&= \mu_{G}(\psi_{G}^{-1}(S))\\
	&= \tilde{\mu}_{G}(S)
	\end{align*}
	A symmetric argument works to show that the other marginal is $\tilde{\mu}_{H}$
\end{proof}

\begin{claim}
	Let $P$ be a Polish space, and $T$ any topological space, and let $f: P \to T$ be continuous. If $\Pi$ is a probability measure on $P$, then $f(\supp{\pi}))$ is dense in $\supp{f_{\ast}(\pi)}$.
\end{claim}
\begin{proof}
	Firstly, take $x \in \supp{\pi}$, and let $V$ be any open neighborhood of $f(x)$. $(f_{\ast}(\pi))(V) = \mu(f^{-1}(V))$, and since $f^{-1}(V)$ is an open neighborhood containing $x$ (which is a point in the support of $\pi$), $f^{-1}(V)$ has positive measure. Thus every neighborhood of $f(x)$ has positive measure. This shows the inclusion
	\[f(\supp{\pi})) \subseteq \supp{f_{\ast}(\pi)}\]
	
	Next, take $y \in \supp{f_{\ast}(\pi)}$, and let $V$ be an open set containing $y$. Then $(f_{\ast} \pi)(V) > 0$. Since $\pi$ is a probability measure on a Polish space, it is Radon, and hence any subset of $P \setminus \supp{\pi}$ has measure zero. Hence $f^{-1}(V)$ must intersect the support of $\pi$. Let $x \in f^{-1}(V) \cap \supp{\pi}$. Then $f(x) \in V \cap f(\supp{\pi})$. Thus, every neighborhood $V$ of $y$ meets $f(\supp{\pi})$, proving density.
\end{proof}

\section{Inverse Problem}
\label{mainresults}
Although barcode transforms are richer invariants than single barcodes, there still exist pairs of metric measure graphs $G$ and $H$ which are not isometric but for which $\mathcal{BT}(G) = \mathcal{BT}(H)$ and $\mathcal{BMT}(G) = \mathcal{BMT}(H)$. The following is a particularly simple example, others can be found in \cite{dey2015comparing}.

	\begin{counterexample}
		\label{counterexample}
	In the following figure \ref{btnotinjfigure}, the lengths of the small branches are all equal, as are the lengths of the middle-sized branches, and finally both central edges have the same length too. For every middle-sized branch in $G$ there is a corresponding branch in $H$ with the same number of small branches, not necessarily on the same side. The barcodes for points on matching branches are the same. Similarly the barcodes for points along the central edges of $G$ and $H$ agree.\\
	
	For any metric graph $X$, we can glue arbitrarily small copies of $G$ and $H$ along the midpoint of their central branches. The resulting pair of graphs will have the same barcode embedding, but will generically not be isometric. Thus, the barcode embedding is not injective on any Gromov-Hausdorff ball.
	
	\begin{center}
	\begin{figure}[htb]
	\labellist
	\small\hair 2pt
	\pinlabel $G$ [t] at 210 207
	\endlabellist
	\centering
	\includegraphics[scale=0.5]{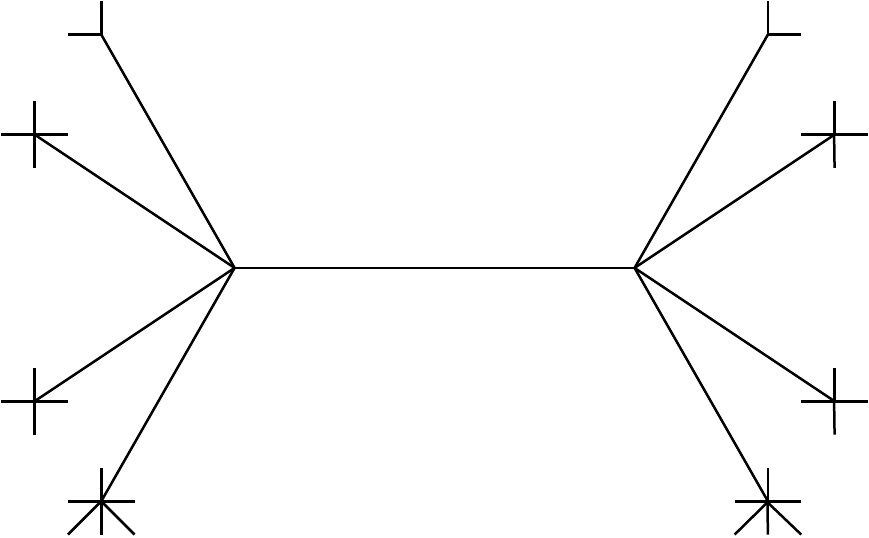} 
		\labellist
	\small\hair 2pt
	\pinlabel $H$ [t] at 210 207
	\endlabellist
	\includegraphics[scale = 0.5]{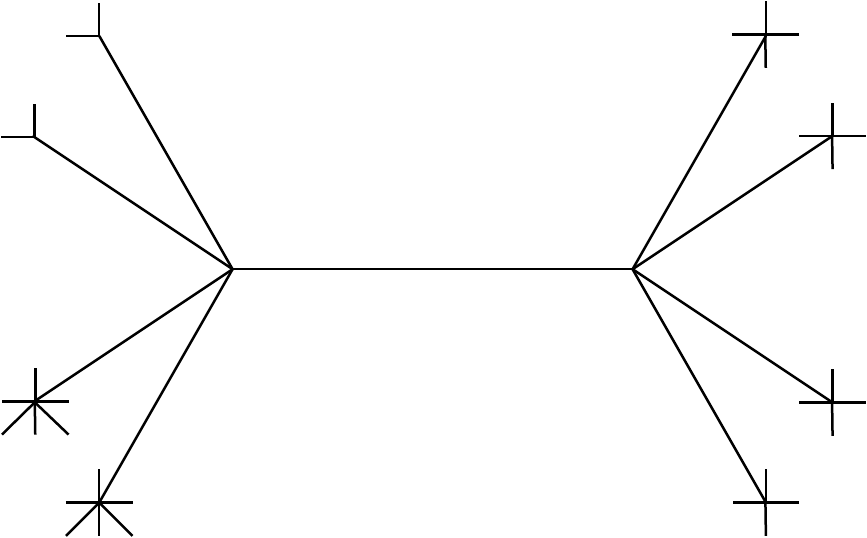}
	
	\caption{$\mathcal{BT}(G) = \mathcal{BT}(H)$, but $G$ and $H$ are not isomorphic.}
	\label{btnotinjfigure}	
	\end{figure}
	\end{center}
	\end{counterexample} 
	
This demonstrates that the barcode transform is not injective on the space $\mathbf{MGraphs}$. That being the case, our approach to the inverse problem is manifold. Firstly, we derive the following local injectivity result, which applies only to the barcode transform.

\begin{theorem}
    \label{btlocalinj}
    $\mathcal{BT}$ is locally injective in the following sense:  $\forall G \in \mathbf{MGraphs}$ there exists a constant $\epsilon(G) > 0$ such that $\forall G' \in \mathbf{MGraphs}$ with $0 < d_{GH}(G,G') < \epsilon(G)$ we have $d_{PD}(G,G') > 0$. 
\end{theorem}

Secondly, we identify a subset of $\mathbf{MGraphs}$ on which the barcode transform and barcode measure transform are injective, and show that it is dense.

\begin{theorem}
    \label{psiinjbtinj}
    The barcode transform and barcode measure transform are injective when restricted to the set $\{G \in \mathbf{MGraphs} \mid \Psi_{G} \operatorname{ injective }\}$.
\end{theorem}

\begin{prop}
\label{densebamboo}
The set $\{G \in \mathbf{MGraphs} \mid \Psi_{G} \operatorname{ injective }\}$ is dense in $\mathbf{MGraphs}$.
\end{prop}

Combining this proposition with the following result from \cite{burago2001course} demonstrates that any compact length space can be approximated by graphs in $\{G \in \mathbf{MGraphs} \mid \Psi_{G} \operatorname{ injective }\}$. This suggests that we can study the structure of more complex geodesic spaces by understanding the barcode transforms of approximating graphs. 

\begin{prop}[\cite{burago2001course},7.5.5]
Every compact length space can be obtained as a Gromov-Hausdorff limit of finite graphs. 
\end{prop}

Next, we show that for most graph topologies, almost every assignment of edge lengths produces a metric graph in $\{G \in \mathbf{MGraphs} \mid \Psi_{G} \operatorname{ injective }\}$. 

\begin{definition}
	Let $X = (V,E)$ be a topological graph. Every vector $v \in \mathbb{R}_{>0}^{E}$ induces a metric structure $G = (X,v)$ on $X$ by assigning edge weights using the entries of $v$ and taking the shortest path metric.
\end{definition}

\begin{theorem}
	\label{btgeninj}
	 Fix a topological graph $X = (V,E)$ with (i) no topological self-loops, and (ii) at least three vertices of valence not equal to two. Let $\mu$ be a measure on $\mathbb{R}_{>0}^{E}$ that is absolutely continuous with respect to the Lebesgue measure. Then, for $\mu$-a.e. vector $v \in \mathbb{R}_{>0}^{E}$,
		\[G = (X,E) \in \{G \in \mathbf{MGraphs} \mid \Psi_{G} \operatorname{ injective }\}\]
\end{theorem}

Lastly, even for those remaining exceptional topologies where injectivity of $\Psi_{G}$ cannot be guaranteed, we demonstrate that restricting ourselves to a full-measure set of edge weights ensures injectivity of the barcode transform.

\begin{theorem}
	\label{btgeninj2}
	 For each topological graph $X = (V,E)$, let $\mu_{X}$ be a measure on $\mathbb{R}_{>0}^{E}$ that is absolutely continuous with respect to the Lebesgue measure. There exists a $\mu_{X}$-full measure subset $\Omega_{X} \subset \mathbb{R}_{>0}^{E}$ such that the following is true: If $X$ and $Y$ are topological graphs, and $v \in \Omega_{X}$ and $w \in \Omega_{Y}$ are length assignments with
	 \[\mathcal{BT}((X,v)) = \mathcal{BT}((Y,w)),\]
	 then $(X,v)$ and $(Y,w)$ are isometric.
\end{theorem}

Thus, restricting ourselves to those metric graphs of the form $(X,v)$ for $X$ a topological graph and $v \in \Omega_{X}$, we can conclude that the persistence distortion distance $d_{PD}$ is a true metric.

\begin{remark}
	For any topological graph $X$, $\Omega_{X} \subset \mathbb{R}_{>0}^{E}$ is the complement of the union of finitely many hyperplanes. Theorems \ref{btgeninj} and \ref{btgeninj2} make use of the Lebesgue-absolutely continuous measures because this ensures that hyperplanes have measure zero.
\end{remark}

\section{Overview of the Proofs of Theorems  \ref{btlocalinj}, \ref{psiinjbtinj}, and \ref{btgeninj}}

The proof of Theorem \ref{btlocalinj} makes use of the analogous result for Reeb graphs from \cite{localequiv}, which can be found in subsection \ref{reebbackground} as Theorem \ref{ReebInequality}.\\

The proof of Theorem \ref{psiinjbtinj} is based on the following result, which is interesting in its own right.

\begin{theorem}
    \label{injinv}
    Let $(G,d_G)$ be a compact, connected metric graph with $\Psi_{G}$ injective. Then $\Psi_{G}$ is an isometry from $(G,d_G)$ to $(\mathcal{BT}(G),\hat{d}_B)$, which is the intrinsic path metric space derived from the metric space $(\mathcal{BT}(G),d_B)$ \footnote{To be precise, an intrinsic path metric is always defined in reference to a class of admissible paths. Here we are considering all $d_B$-continuous paths}. Moreover, if $G$ is also equipped with a measure then this isometry is measure-preserving.
\end{theorem}

This theorem states that when $\Psi_{G}$ is injective we can recover $G$ as a metric graph from $\mathcal{BT}(G)$ or $(G,\mu_G)$ from $\mathcal{BMT}(G)$. Thus, when $G$ and $G'$ are two metric graphs (or metric measure graphs) with $\Psi_{G}$ and $\Psi_{G'}$ injective, the equality of their barcode transforms (or barcode measure transforms) implies equality of the original graphs, demonstrating injectivity. The proof of Theorem \ref{injinv} relies in turn on the following observation.

	\begin{prop}
	\label{localisom}
	If $(G,d_G)$ is a compact, connected, metric graph that is not a circle, then the map $\Psi_{G}: G \to B$ is a local isometry in the following sense. For any fixed basepoint $p \in G$ there is an open neighborhood $N$ such that $\forall q \in N$, $d_G(p,q) = d_{B}(\Psi_{G}(p),\Psi_{G}(q))$.
	\end{prop}
	
	The proof of Proposition \ref{densebamboo} makes use of Theorem \ref{injinv} by constructing, for every metric graph $G$, an approximating sequence $G_{n}$ in the Gromov-Hausdorff metric for which $\Psi_{G_n}$ is injective.\\

	To obtain Theorem \ref{btgeninj}, we combine Theorem \ref{injinv} with the following proposition, and the fact that hyperplanes have Lebesgue measure zero:
	
	\begin{prop}
		\label{genericinj}
		Fix a topological graph $X = (V,E)$ with (i) no topological self-loops, and (ii) at at least three vertices of valence at least three. Then there exists a collection of hyperplanes $\{L_1, \cdots, L_n\} \subset \mathbb{R}_{>0}^{E}$ with the following property: If $v \in \mathbb{R}_{>0}^{E}$ is an edge length such that $v \notin \bigcup_{i=1}^{n} L_{i}$,
		then the associated metric graph $G =(X,v)$ has $\Psi_{G}$ injective.
	\end{prop}

Intuitively, Proposition \ref{genericinj} is a generalization of the result that random metric graphs have trivial automorphisms groups. However, our result is strictly stronger, as it is possible for $\Psi_{G}$ to fail to be injective even if $\operatorname{Aut}(G)$ is trivial, as illustrated in Figure \ref{noinjnoautofigure}.

	\begin{center}
		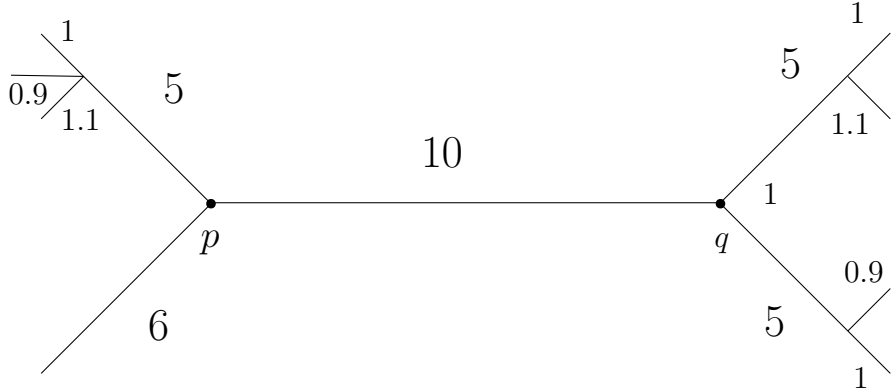
\begin{figure}[htb]
			\labellist
			\small\hair 2pt
			\pinlabel $p$ at 80 55
			\pinlabel $6$ at 59 22
			\pinlabel $1.1$ at 29 101
			\pinlabel $0.9$ at 10 105
			\pinlabel $1$ at 23 132
			\pinlabel $5$ at 65 114
			\pinlabel $10$ at 167 86
			\pinlabel $q$ at 260 55
			\pinlabel $5$ at 290 25
			\pinlabel $1$ at 323 2
			\pinlabel $0.9$ at 327 44
			\pinlabel $1.1$ at 315 100
			\pinlabel $5$ at 293 122
			\pinlabel $1$ at 320 133
			\endlabellist

			\includegraphics{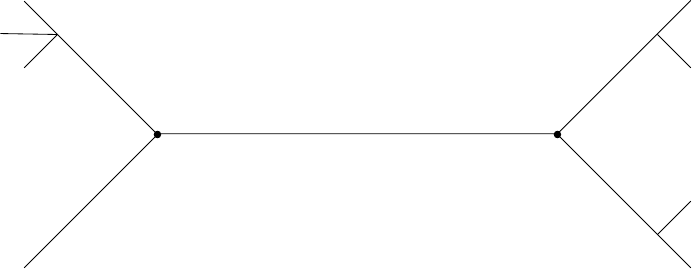}
			\caption{The basepoints $p$ and $q$ produce identical barcodes, despite the graph having a trivial automorphism group.}
						\label{noinjnoautofigure}
		\end{figure}
	\end{center}

    Let us now consider why the above result cannot be stated for \emph{any} graph topology. Any metric graph with a topological self-loop admits an automorphism that flips the loop. Meanwhile, certain combinatorial graphs with fewer than three vertices of valence not equal to two, such as a pair of vertices connected by multiple edges, admit automorphisms flipping those vertices regardless of the metric structure chosen. These automorphisms are obstructions to injectivity of $\Psi_{G}$, so we cannot apply Theorem \ref{injinv} directly.\\
    
    Theorem \ref{btgeninj2} states that, even for these exceptional topologies where injectivity of $\Psi_{G}$ cannot be ensured, the collection of barcodes $\mathcal{BT}(G)$ contains enough information to reconstruct $G$. As before, this requires that we remove from consideration, for each graph topology $X = (V,E)$, finitely many hyperplanes in $\mathbb{R}_{>0}^{E}$. The proof of this proposition is carried out in section \ref{sec: selfloops}, and consists of identifying the ways in $\Psi_{G}$ fails to be injective for these particular graph topologies, and showing they are simple enough to allow $\mathcal{BT}$ to remain injective. Note, however, that the same is not true for the $\mathcal{BMT}$, as it is possible for the failure of injectivity of $\Psi_{G}$ to obscure the original measure on $G$, as demonstrated in the following counterexample.

\begin{counterexample}
Let $(G,\mu)$ be a metric-measure graph homeomorphic to an interval, and let $f: G \to G$ be the isometry exchanging its leaves. Let $S \subset G$ be a measurable subset for which $S \cap f(S) = \emptyset$. Then $(\Psi_{G})_{\ast}(\mu)(\Psi_{G}(S)) = \mu(S) + \mu(f(S))$, as $S$ and $f(S)$ are mapped to the same subset of barcode space by $\Psi_{G}$. The resulting measure on barcode space is thus obtained by symmetrizing $\mu$ with respect to the automorphism $f$, and since there are many distinct measures with the same such symmetrization, this procedure cannot be reversed and $\mathcal{BMT}$ will not be injective.
\end{counterexample}

Detailed proofs are contained in the following sections. The sections are ordered by logical implication and not the order in which their results appear above.

\begin{itemize}
	\item Section \ref{reebproofs} proves Lemma \ref{recovergraph} and Corollary \ref{pointedinj}.
    \item Section \ref{applocalism} proves Proposition \ref{localisom}, that $\Psi_G$ is a local isometry when $G$ is not a circle.
    \item Section \ref{appinjiv} proves Theorem \ref{injinv}, that $G$ can be recovered from $\mathcal{BT}(G)$ when $\Psi_G$ is injective.
    \item Section \ref{sec: bamboographs} proves proposition \ref{densebamboo}, showing density of $\Psi_{G}$-injective graphs.
    \item Section \ref{sec: geninj} proves Theorem \ref{genericinj}, that $\Psi_G$ is generically injective for a large class of measures on $\mathbf{MGraphs}$.
    \item Section \ref{sec: selfloops} discusses the case of topological self-loops and two or fewer vertices of valence not equal to two.
    \item Section \ref{appbtlocalinj} proves Theorem \ref{btlocalinj}, that $\mathcal{BT}$ is locally injective.
\end{itemize}

\section{Proofs of Lemma \ref{recovergraph} and Corollary \ref{pointedinj}}
\label{reebproofs}
We will need the following lemma.
	\begin{lemma}
	\label{epsilonlemma}
	Let $(G,x)$ be a pointed, connected, and compact metric graph, and fix $y \in G$. There is a constant $\epsilon>0$, depending on $x$ and $y$, with the property that for any point $z$ with $d(y,z) < \epsilon$, we have $d(y,z) = |d(x,y) - d(x,z)|$. In other words, the distance between $y$ and $z$ can be written as the difference of their distances to $x$.
	\end{lemma}
	\begin{proof}
	When $y=z$, the result holds trivially, so we now assume $y \neq z $.Consider the depiction of a metric graph in Figure \ref{fig:graphepsilon}. For $\epsilon$ small enough, the finiteness of our graph and the condition $d(y,z) < \epsilon$ forces $z$ to sit on an edge (or half-edge, if $y$ is not a vertex) adjacent to $y$. For each such edge, there is either a geodesic from $y$ to $x$ that moves along that edge, or there is not. In the former situation, moving along that edge (which has some positive length) will bring one closer to $x$, and the geodesic from $y$ to $x$ is exactly the concatenation of the geodesic from $y$ to $z$ and the geodesic from $z$ to $x$, giving $d(y,z) = |d(x,y) - d(x,z)|$. In the latter situation, for $\epsilon$ sufficiently small, the geodesic from $z$ to $x$ must pass through $y$, and we obtain the same equality. Lastly, since our graph is finite, there are only finitely many edges adjacent to $y$, and some sufficiently small $\epsilon$ works for points $z$ sitting on every edge adjacent to $y$.s
		\end{proof}
	\begin{center}
	\begin{figure}[htb]
		\labellist	
		\small\hair 2pt
		\pinlabel $x$ at 33 26
		\pinlabel $z_3$ at 113 26
		\pinlabel $y$ at 110 68
		\pinlabel $\epsilon$ at 102 80
		\pinlabel $z_1$ at 112 88
			\pinlabel $z_2$ at 91 59
		\endlabellist
	    \centering
	\includegraphics[scale=1]{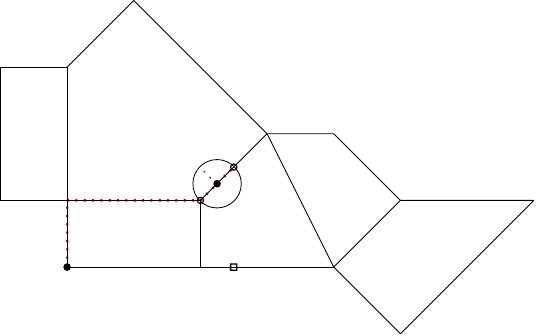}
	\caption{A metric graph with basepoint $x$ and fixed $y$. For $\epsilon$ small enough, $d(y,z) < \epsilon$ implies that $z$ sits on an edge adjacent to $y$, as with $z_1$ or $z_2$. The geodesic from $x$ to $y$ passes through $z_2$, implying the claimed equality for $z_2$. Similarly, the geodesic from $x$ to $z_1$ passes through $y$, and so we can also deduce the equality for $z_1$. Neither statement is true when comparing the geodesics from $x$ to $y$ and $x$ to $z_3$ respectively, and indeed $d(y,z_3) \neq |d(x,y) - d(x,z_3)|$.}
	    \label{fig:graphepsilon}
	\end{figure}
	\end{center}

\begin{proof}[Proof of Lemma \ref{recovergraph}]

	We will show that, given any path $\gamma$, the length of this path in both $d_{G}$ and $d_{f}$ is the same. Since both $d_{G}$ and $d_{f}$ are intrinsic metrics, this will imply that they are equal. Given such a path $\gamma: I \to G$, let $I' \subset I = [0,1]$ be the largest initial interval such that $d_{G}$ and $d_{f}$ agree on $\gamma(I')$. We claim that $I' = I$. Suppose not, and let $s = \sup I'$. The continuity of our metrics implies that $s \in I'$, and hence $I' = [0,s]$ for $s<1$.\\
	
	Now, applying Lemma \ref{epsilonlemma} to $y = \gamma(s)$, and making use of the continuity of the map $\gamma$, we find that we can extend $I'$ to a strictly larger interval $I''$, where $d_{G}$ and $d_{f}$ agree on $\gamma(I'')$. This contradicts our choice of $I'$, so we conclude that $I' = I$ and $s=1$. 
	\end{proof}
	
	\begin{proof}[Proof of Corollary \ref{pointedinj}]
	 Let $(G,x)$ and $(G',x')$ be a pair of pointed metric spaces, with associated Reeb graphs $R_f =\Phi(G,x)$ and $R_g\Phi(G',x')$. If $d_{FD}(R_f,R_g) = 0$ then by Lemma \ref{FDzero}, there is a Reeb-graph isometry $\phi$ from $R_f$ to $R_g$ (in particular, a homeomorphism) preserving their respective height functions. Thus for any continuous path $\gamma \subset G$, the integral $\int_{\gamma} |df|$ is equal to $\int_{\phi(\gamma)} |dg|$. Since these integrals correspond to the lengths of these paths in $d_{G}$ and $d_{G'}$ respectively, and since $d_{G}$ and $d_{G'}$ are both length metrics, we see that $\phi$ induces an isometry between $(G,d_G)$ and $(G',d_{G'})$. This isometry must preserve basepoints since $x$ and $x'$ are the unique points at ``height" zero, and hence $\phi(x) = x'$.
	\end{proof}

\section{Proof of Proposition \ref{localisom}: $\Psi_{G}$ is a local isometry}
\label{applocalism}

The proof of Proposition \ref{localisom} relies on the following lemma.

\begin{lemma}
\label{birthdeathvary}
Let $G$ be any compact metric graph that is not a circle. Then, for every choice of basepoint $p \in G$, there exists a neighborhood radius $\epsilon(p) > 0$ and a non-diagonal interval $I = (a,b)$ in the extended persistence barcode $\Psi_{G}(p)$ such that, for any $q$ with $d = d(p,q) < \epsilon(p)$, there is a corresponding interval $I' \in \Psi_{G}(q)$ with $d_{\infty}(I,I') = d$ and such that $\operatorname{pers}(I') \geq \operatorname{pers}(I)/2$.
\end{lemma}

Assuming this lemma for the moment, let us prove the proposition.

\begin{proof}[Proof of Proposition \ref{localisom}]
		Fix a basepoint $p \in G$, and take a nearby basepoint $q$, writing $d = d(p,q)$. Theorem \ref{perstable} implies that $d_{B}(\Psi_{G}(p),\Psi_{G}(q)) \leq d$.\\
	
	For the reverse inequality, let $\delta > 0$ be the minimum value among both the distances between non-diagonal intervals in $\Psi_{G}(p)$ (considered without multiplicity, and comparing intervals in the same dimension only) as well as the set of values $\{\operatorname{pers}(I)\}$ for those same non-diagonal intervals. Take $q$ to be a point close enough to $p$ so that it satisfies the claim of Lemma \ref{birthdeathvary}, and let us further assume that $d(p,q) < \delta /2$ if it is not already.\\
	
	Let $I \in \Psi_{G}(p)$ and $I' \in \Psi_{G}(q)$ be as given in Lemma \ref{birthdeathvary}. We know that in an optimal matching, $I'$ must be paired with a (potentially diagonal) interval $\hat{I} \in \Psi_{G}(p)$ with $d_{\infty}(\hat{I},I') \leq d$. We claim that $\hat{I} = I$, so that $d_{\infty}(I',\hat{I}) = d$, and hence $d_{B}(\Psi_{G}(p),\Psi_{G}(q)) \geq d$.\\
	
	The condition $d_{\infty}(I,I') = d < \delta /2$ implies that $d_{\infty}(\hat{I},I') \geq \delta /2 > d$ for all non-diagonal $\hat{I} \in \Psi_{G}(p)$ distinct from $I$. Moreover, the condition $\operatorname{pers}(I') \geq \operatorname{pers}(I)/2$, taken together with the fact that $\operatorname{pers}(I) > \delta$, implies that the distance from $I'$ to the diagonal is at least $\delta/2 > d$ as well. Thus, the only possibility for $\hat{I}$ is $I$ itself.
	\end{proof}

We now turn to the proof of Lemma \ref{birthdeathvary}. We will make use of the following two results from \cite{dey2015comparing}. These results concern the behavior of a point in the persistence diagram as a basepoint is varied along the arc-length parametrization $s:[0,\operatorname{Len}((\sigma))] \to G$ of an edge $\sigma$. The notation $P_{s(t)}$ refers to the persistence diagram associated to the basepoint $s(t)$. The authors show that a point in the persistence diagram $P_{s(0)}$ can be tracked as one moves along $\sigma$, and that its birth time $b(t)$ and death time $d(t)$ evolve in a Lipschitz, controlled way:

%\begin{remark}
%	\label{rmk:hyp}
%Proposition \ref{prop:deydeath} assumes that, for all $t \in [0,\operatorname{Len}((\sigma))]$, the death time $d(t)$ corresponds to the distance from $s(t)$ to some upfork of valence at least three (as opposed to the possibility that $d(t_{0})=0$ for some point $s(t_{0})$ with valence at most two). This hypothesis is made clear in \S 5.1.1 of \cite{dey2015comparing}, as a means of simplifying the proofs of that section.
%\end{remark} 

\begin{prop}[\cite{dey2015comparing}, Proposition 14]
	\label{prop:deydeath}
	Fix a persistence point $(b(0), d(0)) \in P_{s(0)}$, and consider the corresponding death-time function
	$d : [0, \operatorname{Len}((\sigma))] \to \mathbb{R}$. Suppose that $d(t)>0$ whenever $s(t)$ is a point of valence less than three, i.e. $s(t)$ is a leaf or sits on the interior of an edge\footnote{This hypothesis is made clear in \S 5.1.1 of \cite{dey2015comparing}, as a means of simplifying the proofs of that section.}. Then $d(t)$ is piecewise linear with at most $O(n)$ pieces, and each linear piece has slope either $1$ or $-1$. This also implies that the function $d$ is $1$-Lipschitz.
\end{prop}

%\begin{corollary}
%	\label{cor:deyfix}
%	Let $p \in G$ be a basepoint whose Reeb graph $\Phi_{G}(p)$ contains an upfork $u$ of valence at least three. Then there is an small neighborhood parameter $\epsilon$ for which the following is true.  Let $\sigma:[0,\epsilon]$ be a parametrization of a small ray starting at $p$, i.e. $\sigma(0) = p$, and moving along an adjacent edge. If $(b(0),d(0))$ is the persistence point corresponding to the upfork $u$, then the death-time function $d(t)$ is linear on $[0,\epsilon]$ with slope $+1$ or $-1$. 
%\end{corollary}
%\begin{proof}
%If $d(0)>0$, then the continuity of $d(t)$ implies that, for some $\epsilon$ sufficiently small, $d(t) > 0$ for $t \leq \epsilon$. This means that, for all $t \leq \epsilon$, the corresponding upfork has valence at least three. We can therefore apply Proposition \ref{prop:deydeath} to deduce the result.\\
%
%If $d(0) = 0$, then $p=u$, so that $p$ has valence at least three. Let $v$ be the closest vertex of valence at least three to $u$ Again, by continuity of the death-time function, there exists $\epsilon$ small enough such that, for $t \leq \epsilon$, $d(t)<d(u,v)/2$. If we shrink $\epsilon$ further such that $\epsilon < d(u,v)/2$, then the only   
%\end{proof}	

\begin{prop}[\cite{dey2015comparing}, Proposition 18]
		\label{prop:deybirth}
	For a fixed persistent point $(b(0), d(0)) \in P_{s(0)}$, the birth-time function
	$b : [0, \operatorname{Len}((\sigma))] \to \mathbb{R}$, tracking the birth time $b(0)$, is piecewise linear with at most $O(m)$ pieces, and each linear piece has slope either $1$,$-1$. or $0$. This also implies that the function $d$ is $1$-Lipschitz.
\end{prop}

\begin{proof}[Proof of Lemma \ref{birthdeathvary}]
To begin, suppose that the Reeb graph $\Phi_{G}(p)$ contains an upfork $u$ of valence at least three, with $d(u,p) > 0$, and let $I = (a,b)$ be the corresponding non-diagonal interval, so that $b>0$. By continuity of the death-time function, as the point $p$ is varied in a small neighborhood, the death time of this interval remains positive, so that the hypothesis of Proposition \ref{prop:deydeath} is satisfied. Propositions \ref{prop:deydeath} and \ref{prop:deybirth} then imply that, for $q$ sufficiently close to $p$, $|d-d'|$ is one of $\{\pm d_{G}(p,q)\}$, and $|b-b'|$ is one of $\{\pm d_{G}(p,q),0\}$. This implies that $d_{\infty}(I,I) = d_{G}(p,q)$. If we further restrict $d_{G}(p,q) < \frac{1}{2}|a-b| = \operatorname{pers}(I)$, we have
	\begin{align*}
|\operatorname{pers}(I') - \operatorname{pers}(I)| & = \frac{1}{2}||a'-b'| - |a-b||\\
& \leq \frac{1}{2}|a-a'| + \frac{1}{2}|b-b'| \\
&  \leq \frac{1}{8}|a-b| + \frac{1}{8}|a-b|\\
& = \frac{1}{4}|a-b|\\
& = \frac{1}{2}\operatorname{pers}(I)
\end{align*}
verifying the condition that $\operatorname{pers}(I') \geq \operatorname{pers}(I)/2$.\\

Suppose now that the Reeb Graph $\Phi_{G}(p)$ contains no upforks distinct from $p$ itself, so that there are no nonzero death times in $\Psi_{G}(p)$, and we cannot directly make use of Proposition \ref{prop:deydeath}. We split our analysis into three cases: (1) $G$ contains no vertices of valence at least three, (2) $p$ itself has valence at least three, or (3) $G$ contains vertices of valence at least three, but $p$ has valence at most two. We will analyze each case in turn.\\

Case (1): If $G$ contains no vertices of valence at least three, and is not a circle, it must be (up to removal of valence-two vertices) a line segment. For $G$ a line segment, the zero-dimensional part of $\Psi_{G}(p)$ consists of the interval $I = (a,0)$, where $a$ is the distance from $p$ to the furthest leaf vertex. The value $a$ changes linearly with slope $1$ or $-1$, switching slope at the center of the interval. Thus, although Proposition \ref{prop:deydeath} does not apply in this case (since the death time is constant), we have ruled out the slope $0$ possibility in Proposition \ref{prop:deybirth}. Hence, as we vary the basepoint $p$ in a small neighborhood, we still obtain $d_{G}(I,I') = d_{G}(p,q)$ and $\operatorname{pers}(I') \geq \operatorname{pers}(I)/2$.\\

%Case (2): If $p$ is a vertex with valence at least three, then it contains at l

Case (2): We focus on one-dimensional homology. As $p$ has valence at least three, $\Psi_{G}(p)$ contains at least two off-diagonal intervals with death time zero, $I_1 = (a_1,0)$ and $I_{2} = (a_2,0)$\footnote{See Lemma \ref{lem:detectvalence} for a detailed proof of this intuitive claim.}. Let $\delta$ be the length of the shortest edge adjacent to $p$, and consider a basepoint $q\in G$ with $d_{G}(p,q) < \delta/2$, as in Figure \ref{fig:localisom_case2}. Let $I_{1}'$ and $I_{2}'$ be the corresponding intervals for $I_{1}$ and $I_{2}$ in $\Phi_{G}(q)$, as in \cite{dey2015comparing}. Supposing that the death time $b_1'$ of $I_{1}' = (a_1',b_1')$ is not equal to zero, it corresponds to the distance from $q$ to an upfork $u$. Since the birth and death times of intervals vary with Lipschitz constant one as the basepoint is moved, $d_{B}(I_{1},I_{1}') \leq d_{G}(p,q)$, and hence $b_1' < \delta /2$. This implies that $u = p$, as $q$ is at distance greater than $\delta /2$ from any other vertex of valence at least three, and so $b_{1}' = d_{G}(p,q)$. Hence $d_{B}(I_{1},I_{1}') \geq d_{G}(p,q)$, from which we can deduce that $d_{B}(I_{1},I_{1}') =d_{G}(p,q)$. Taking $d_{G}(p,q) < \frac{1}{2}|b_1  - a_1|$, we can further ensure that $\operatorname{pers}(I_{1}') \geq \operatorname{pers}(I_{1})/2$.\\

To justify the supposition that $I_{1}'$ has nonzero death time, observe that for $q$ sufficiently close to $p$, $q$ has valence two and hence at most one one-dimensional persistent point with death time zero. Thus it is not possible for both $I_{1}'$ and $I_{2}'$ to have death time zero, and up to relabeling we can choose $I_{1}'$ to have strictly positive death time. \\

Case (3): Since $p$ has valence less than three, sufficiently small neighborhoods of $p$ are homeomorphic to intervals. Let $r$ be the largest value such that the open ball of radius $r$ at $p$ is homeomorphic to an interval (this is finite, as $G$ contains a vertex of valence at least three, and hence cannot be an interval). Let $x$ and $y$ be the endpoints of this interval of radius $r$. As the homeomorphism type of our neighborhood changes at $r$, we either find that $x=y$, and the neighborhood becomes a circle, or $x\neq y$, and one of the endpoints branches out, producing an upfork of valence at least three in $\Psi_{G}(p)$. By hypothesis, we can rule out the second possibility. Now, $x=y$ must be antipodal to $p$ on our circle of circumference $2r$, and, since $G$ is not homeomorphic to a circle, $x=y$ must have valence at least three. We know that $x=y$ cannot have valence of four or more, as this would make it an upfork. Thus, $x=y$ has valence exactly three. Now, any other vertex $v \in G$ connected to $p$ through $x=y$ must have valence at most two, as it would otherwise be an upfork. Hence we can conclude that $G$ looks (up to deletion of valence-two vertices) as in Figure \ref{fig:localisom_case3}. In this case, we turn to zero-dimensional homology; the zero-dimensional part of $\Psi_{G}(p)$ contains a single interval $I = (0,d(p,v))$, corresponding to the distance from $p$ to the leaf vertex $v$. It is easily seen that for $q$ sufficiently close to $p$, the zero-dimensional part of $\Psi_{G}(q)$ consists of $I' = (0,d(q,v)) = (0,d(p,v) - d(p,q))$. Thus $d_{\infty}(I,I') = d(p,q)$. Lastly, if $d(p,q) < d(p,v)/2$, then $\operatorname{pers}(I') \geq \operatorname{pers}(I)/2$.

\end{proof}	

	\begin{center}
	\begin{figure}[htb]
					\labellist
		\small\hair 2pt
		\pinlabel $p$ at 73 87
		\pinlabel $q$ at 100 97
		\pinlabel $\delta$ at 65 40
		\endlabellist
		\includegraphics[scale=0.8]{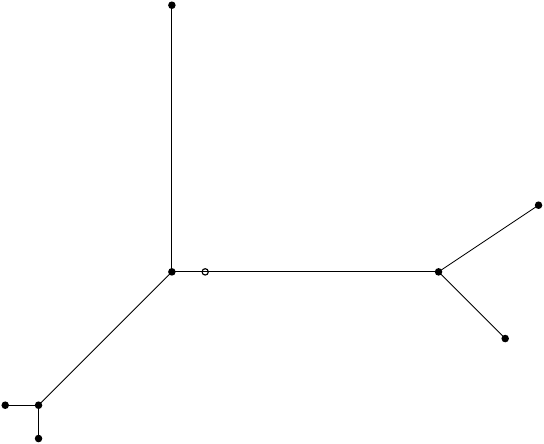}
		\caption{The vertex $p$ has valence at least three, and $q$ sits on an edge adjacent to $p$, strictly closer to $p$ than any other vertex of valence at least three.}
		
		\label{fig:localisom_case2}
	\end{figure}
\end{center}

	\begin{center}
	\begin{figure}[htb]
		\labellist
		\small\hair 2pt
		\pinlabel $p$ at 68 -5
		\pinlabel $q$ at 101 4
		\pinlabel $x=y$ at 65 115
		\pinlabel $v$ at 65 205
		\endlabellist
		\includegraphics[scale=0.8]{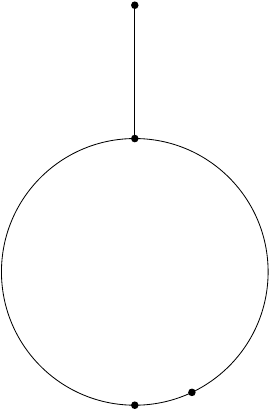}
		\caption{The graph $G$ looks as above, with $p$ antipodal to the vertex $v$. For $q$ close to $p$, the furthest point on $G$ is the leaf vertex $v$.}
		
		\label{fig:localisom_case3}
	\end{figure}
\end{center}

\section{Proof of Theorem \ref{injinv}: Recovering $G$ when $\Psi_{G}$ is injective}
 \label{appinjiv}

	Let us first consider the map $\mathcal{BT}$. Given $G \in \mathbf{MGraphs}$ with $\Psi_{G}$ injective, endow $\mathcal{BT}(G)$ with the intrinsic metric $\hat{d}_{B}$ induced by the barcode metric. As $\Psi_{G}$ is an injective continuous map from a compact space to a Hausdorff space (recall that $\mathcal{BT}(G)$ in Hausdorff for the metric $d_{B}$ by Lemma \ref{bottleneckmetric}), it is a homeomorphism on to its image. For the remainder of the proof, we identify $G$ with $\mathcal{BT}(G)$ via the map $\mathcal{BT}$, considering $d_{G}, d_{B}$, and $\hat{d}_{B}$ as metrics on $G$. Moreover, $\Psi_{G}$ being a homeomorphism implies that $d_{B}$ and $d_{G}$ induce the same topology on $G$, so the class of $d_{B}$-continuous paths is the same as the class of $d_{G}$-continuous paths.\\
	
% 	Let $\gamma: I \to G$ be a $d_{G}$-continuous path, and $P = \{0 = t_0, \cdots, t_n = 1\}$ a partition for $\gamma$. Let $\ell_{G,P}(\gamma)$ and $\ell_{B,P}(\gamma)$ denote the lengths of $\gamma$ in $d_{G}$ and $d_{B}$ with respect to this partition. Lemma \ref{psilip} implies that
	
% 	\[ \ell_{G,P}(\gamma) = \sum_{i=0}^{n-1} d_{G}(\gamma(t_i),\gamma(t_{i+1})) \geq \sum_{i=0}^{n-1} d_{B}(\gamma(t_i),\gamma(t({i+1})) \geq \ell_{B,P}(\gamma)\]
	
% 	Since this inequality holds for all partitions $P$, we see that the length of $\gamma$ in $G$, $\ell_{G}(\gamma)$ is greater than or equal to $\ell_{B}(\gamma)$, the lenth of $\gamma$ in $B$. Moreover, $d_{B}(\gamma) \geq \hat{d}_{B}(\gamma(0),\gamma(1))$, as the latter quantity is computed as the infimum of the $d_{B}$-lengths of all $d_{G}$-continuous paths, and here we have just taken one. Since the inequality $d_{G}(\gamma) \geq \hat{d}_{B}(\gamma(0),\gamma(1))$ holds for any $d_{G}$-continuous path $\gamma$, and since $d_{G}$ is an intrinsic metric, we may conclude that for any $p,q \in G$ that $d_{G}(p,q) \geq \hat{d}_{B}(p,q)$. The remainder of the proof is thus devoted the reverse inequality.\\
	
	Let $\gamma: I \to G$ be a $d_{G}$-continuous path, and $P = \{0 = t_0, \cdots, t_n = 1\}$ a partition. Let $\ell_{G,P}(\gamma)$ and $\ell_{B,P}(\gamma)$ denote the lengths of $\gamma$ in $d_{G}$ and $d_{B}$ with respect to this partition. We claim that $P$ admits a refinement $P \subseteq P' = \{0 = r_0, \cdots, r_{m} = 1\}$ for which $\ell_{G,P'}(\gamma) = \ell_{B,P'}(\gamma)$. As the length of a path in a metric space is the supremum of the lengths of its partitions, considered over the set of all possible partitions, this implies that $\gamma$ has the same length in both $d_{G}$ and $d_{B}$. Since $\hat{d}_{B}$ is the intrinsic metric defined using $d_{G}$-continuous paths and $d_{G}$ is an intrinsic metric, this will imply that $\hat{d}_{B}(p,q) = d_{G}(p,q)$ for all $p,q \in G$.\\
	
	For each time $t \in I$ and corresponding point $\gamma(t) \in G$, there is a constant $\epsilon_{t}$ witnessing the validity of Proposition \ref{localisom} for $\gamma(t)$. Let $U_{t}$ be the open $d_{G}$-neighborhood of $\gamma(t)$ of radius $\epsilon_{t}$. Since $\gamma: I \to G$ is continuous, there is a constant $\delta_{t}$ such that $\gamma((t-\delta_t, t + \delta_t)) \subseteq U_t$. Let $V_t = (t-\delta_t/2, t + \delta_t/2)$. The sets $V_t$ form an open cover of $I$, and hence by compactness a finite subcover exists, corresponding to a collection of times $\Omega = \{r_{1}, \cdots, r_{k}\}$. Let us augment $\Omega$ with the times in $P$ to produce our refinement $P'$. Note that if $r,r' \in P'$ are two consecutive times, the triangle inequality implies $r' - r < \max \{\delta_r, \delta_r'\}$, so that either $\gamma(r') \in \gamma((r-\delta_r, r + \delta_r)) \subseteq U_r$ or $\gamma(r) \in \gamma((r'-\delta_r', r' + \delta_r')) \subseteq U_{r'}$. Direct application of Proposition \ref{localisom} then yields
	
\[ \ell_{G,P}(\gamma) = \sum_{i=0}^{m-1} d_{G}(\gamma(r_i),\gamma(r_{i+1})) = \sum_{i=0}^{m-1} d_{B}(\gamma(r_i),\gamma(r({i+1})) = \ell_{B,P}(\gamma)\]
	
	completing the proof that $d_{G} = \hat{d}_{B}$. For the map $\mathcal{BMT}$, we can obtain $\mathcal{BT}(G)$ by taking the support of the pushforward measure, since we are by assumption working with measures of full support. As we have just seen, we can then obtain the underlying metric graph $G$. The measure of a Borel subset $S \subseteq G$ is then equal to $(\Psi_{G})_{\ast}(\mu)(\Psi_{G}(S))$.

\section{Proposition \ref{densebamboo}: Density of $\Psi_{G}$-injective Graphs}
\label{sec: bamboographs}

Given a compact metric graph $G$ with vertex set $V$, we define a \emph{cactus approximation} of $G$ as follows. Let $S \subset G$ be a finite set of points containing $V$. Define $\omega(S) = \max_{p,p' \in S} d_{G}(s,s')$,  $\delta(S) = \min_{p,p' \in S} d_{G}(s,s')$, and $\delta(V) = \min_{v,v' \in V} d_{G}(v,v')$. Our intention is to attach small edges to $G$ along the points in $S$ (but not the leaf vertices). To that end, let $\alpha : S \to \mathbb{R}_{+}$ be any function that is zero on the leaves of $G$. We will produce a new graph by attaching to each point $p \in S$ an interval $I_p$ of length $\alpha(p)$. The resulting graph $H$ will be denoted by $\operatorname{Cactus}_{S,\alpha}(G)$. The idea is illustrated in Figure \ref{fig:cactus}. The points in $S$ are drawn in red, and to each one we have attached a new interval, which we will call a \emph{thorn} in the remainder of the proof. Note that the thorns attached to leaves in $G$ have length zero, and hence do not add anything to the graph. Additionally, since $S$ is finite, $H$ is composed of finitely many edges and vertices, and hence is still an element of $\mathbf{MGraphs}$.
	\begin{figure}[htb]
				\labellist
		\small\hair 2pt
		\pinlabel $p$ at 91 334
		\pinlabel $I_p$ at 77 297
		\endlabellist
	    \centering
	\includegraphics[scale=0.6]{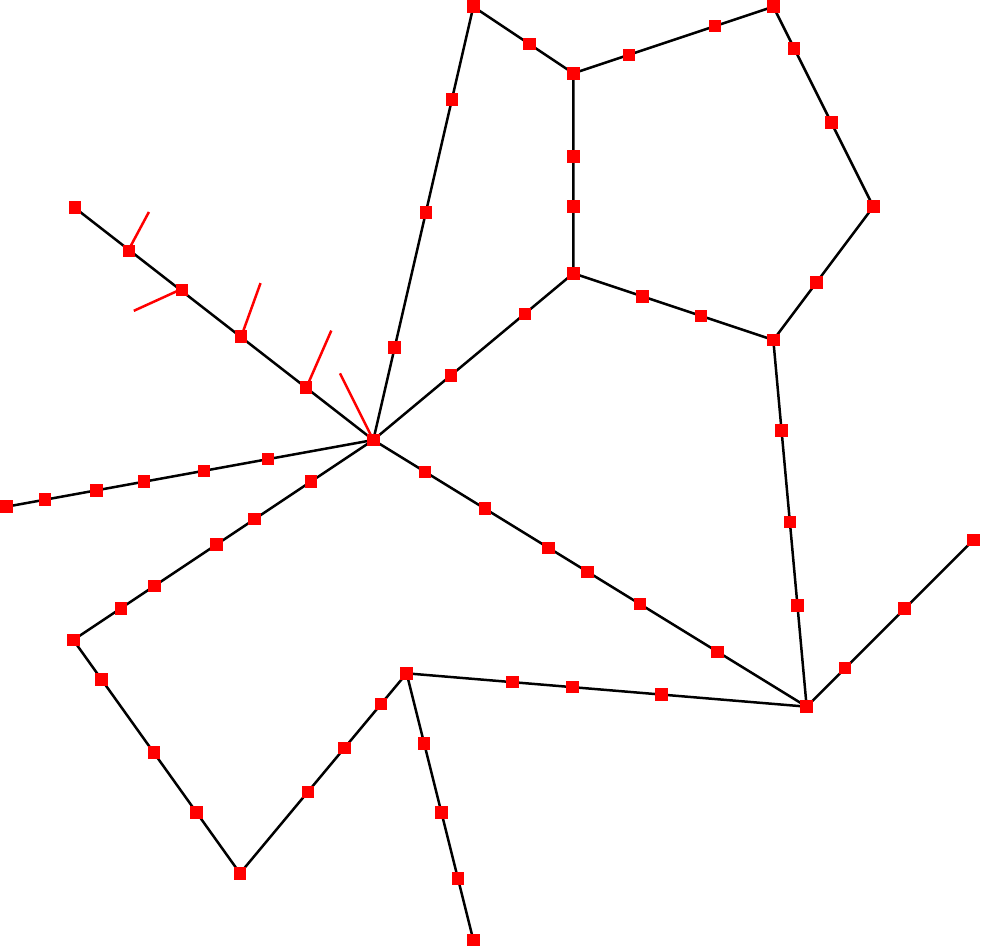}
	    \caption{The cactus approximation of $G$.}
	    \label{fig:cactus}
	\end{figure}

	It is clear that if $\alpha(p) \leq \epsilon$ for all $p \in S$, the Gromov-Hausdorff distance between $G$ and any $\operatorname{Cactus}_{S,\alpha_{p}}(G)$ is at most $\epsilon$. The proof of Proposition \ref{densebamboo} follows from Lemma \ref{bambooinjective}.
	
	\begin{definition}
	Let $G$ be a compact graph. The \emph{injectivity radius} of $G$, denoted $\operatorname{inj}(G)$, is half the length of the shortest closed curve on $G$ (i.e. half the length of the \emph{systole} of $G$).
	\end{definition}
	
	\begin{lemma}
	\label{bambooinjective}
	Let $G$ be any compact metric graph. Let $S \subset G$ be a finite subset and take $0 < \alpha < \delta(S) /2$. Suppose that $2\omega(S) + \alpha < \operatorname{inj}(G)$, $\omega(S) < \delta(V)$, and that $\alpha : S \to \mathbb{R}_{+}$ is injective and nonzero on $S' = S \setminus \{ \mbox{leaves in $G$} \}$. Then if $H = \operatorname{Cactus}_{S,\alpha}(G)$, $\Psi_{H}$ is injective.
	\end{lemma}
	
	\begin{proof}
	
		We claim that the location of a point $x \in H$ can be determined from its barcode.\\
		
	\paragraph*{\bf Determining if a point $x$ lies on a thorn from $\Psi_{H}(x)$.}
	
 First of all, we can determine from $\Psi_{H}(x)$ whether or not $x$ sits on a thorn. If $x$ is on a thorn $I_p$, the distance from $x$ to the tip of this thorn is less than $\alpha$. The distance from $x$ to the tip of any other thorn is at least $\delta(S) > \alpha$. Moreover, since $\alpha < \operatorname{inj}(G)$, the ball around $x$ of radius $\alpha$ contains no loops in $G$. This implies that the smallest positive birth time of a one-dimensional interval in $\Psi_{H}(x)$ is the distance from $x$ to the tip of $I_p$, and the corresponding interval dies at $0$. The first smallest death time in $\Psi_{H}(x)$ is necessarily the distance from $x$ to the base of $I_p$, which is an upfork in $\Phi_{H}(x)$.\\
	
	Let $x \in H$ be a point not sitting on a thorn. Suppose first that $x \in H$ sits on an interior edge of $H$, i.e. neither boundary vertex of this edge is a leaf. As before, the $\alpha$-ball around $x$ contains no cycles in $G$, and moreover since $\alpha < \delta(S)$ it contains no leaf of $G$, so if $\Psi_{H}(x)$ contains an interval born before $\alpha$ it corresponds to the distance from $x$ to the tip of a thorn, and such a barcode dies at the distance from $x$ to the base of that thorn, which is nonzero. Suppose next that $x \in H$ sits a leaf edge. An additional possibility from the prior case is that there is an interval in $\Psi_{H}(x)$ born before $\alpha$ that corresponds to the distance from $x$ to a leaf in $G$. In that case the distance from $x$ to the other vertex on its edge is at least $\delta(S) - \alpha > 2\alpha - \alpha = \alpha$, and this is the first nonzero death time in $\Psi_{H}(x)$.\\
	
	Thus we have seen that the barcodes associated to points on thorns are distinct from those that are not on thorns, and hence we can identify whether a point $x \in H$ lies on a thorn from its barcode.\\
	
	\paragraph*{\bf Determining the location of a point $x$ from $\Psi_{H}(x)$.}

    When $x$ lies on a thorn $I_p$, we have seen that $\Psi_{H}(x)$ records its distance from the tip and base of that thorn. From this information, the length $\alpha(p)$ of the thorn can also be derived. Since $\alpha$ is injective on $S'$, we can identify which thorn this is, and where on it $x$ sits.\\
	
	Next, suppose that $x$ does not sit on a thorn. If $x$ lies on an interior edge, with endpoints $p_1, p_2 \in S$, we claim that we can deduce what these points $p_1$ and $p_2$ are from $\Psi_{H}(x)$, as well as the distance from these points to $x$. This uniquely identifies the point $x$, as $\omega(S) < \operatorname{inj}(G)$ implies that there is only one point equidistant between $p_1$ and $p_2$. To do this, observe that the ball of radius $\omega(S) + \alpha$ around $x$ contains the thorns on either end of the edge. Since $\omega(S) + \alpha < \operatorname{inj}(G)$, this ball is a tree and does not contain any loops in $G$. Thus the intervals in $\Psi_{H}(x)$ born before $\omega(S) + \alpha$ correspond to the distances from $x$ to the tips of these thorns, and the death times of the corresponding intervals are the distances from $x$ to the base of those thorns. The barcode may also contain intervals corresponding to the distances from $x$ to leaves in $G$, but such intervals die at zero and cannot be confused with intervals coming from thorns. From this information, one can further deduce the length of those thorns, and hence the thorns themselves and their corresponding bases in $S$. This uniquely identifies the point $x \in H$.\\
	
	Lastly, suppose $X$ sits on a leaf edge, with non-leaf endpoint $p$. Let $q_1, \cdots, q_n$ be all the vertices in $S$ adjacent to $p$. Since $\delta(S) < \omega(S)$, it is not possible for two distinct leaves in $G$ to be adjacent to the same vertex in $S$, and hence all of the $q_{i}$ are in $S'$. As before, we claim that we can identify the points $p,q_1, \cdots, q_n$, as well as the distances from these points to $x$, from $\Psi_{H}(x)$. This will uniquely identify the point $x$. To do this, observe that the ball of radius $2\omega(S) + \alpha$ around $x$ contains the thorns $I_p, I_{q_1}, \cdots, I_{q_n}$. As before, we can identify which thorns these are from $\Psi_{H}(x)$, and subsequently deduce the distances from their bases to $x$.
	\end{proof}

\section{Proposition \ref{genericinj}: Generic Injectivity of $\Psi_{G}$ for $\mathbf{MGraphs}$}
\label{sec: geninj}

We will prove the result by contradiction. That, is, we will show that there is a positive integer $C(X)$, depending on the topological graph $X$, such that if $\Psi_{G}$ is \emph{not} injective then the edge lengths of $G$ satisfy a nontrivial\footnote{By nontrivial, we mean that the corresponding solution space is not all of $\mathbb{R}_{>0}^{E}$, but a proper hyperplane. Equivalently, the equation should not hold for every assignment of edge weights.} linear equation with integer coefficients in the range $[-C(X), C(X)]$. There are finitely many such linear equations, corresponding to finitely many hyperplanes, and avoiding these hyperplanes guarantees injectivity of $\Psi_{G}$.\\

To simplify the following proof, we will remove any vertices of degree two and merge the adjacent edges; this may introduce multiple edges betwen vertices, which is fine for our purposes. Note that if the edge lengths of $G$ do not satisfy any nontrivial linear equalities with coefficients in $\{-C(X), \cdots, C(X) \}$, then neither will the edge lengths of this new graph. Thus we assume from the beginning that every vertex in $G$ either has valence greater than or equal to three or has valence equal to one (a leaf vertex).\\

Before we proceed with the case analysis, we introduce some useful lemmas.
    
    \begin{lemma}
    \label{lem:detectvalence}
    Let $G$ be any compact metric graph, possibly with self-loops. For any basepoint $p \in G$, it is possible to deduce the valence of $p$ from $\Psi_{G}(p)$, where the valence of a non-vertex point is considered to be $2$. Indeed, the valence of $p$ is one less than the number of intervals of the form $(0, \cdot)$ in the one-dimensional part of $\Psi_{G}(p)$.
    \end{lemma}
    \begin{proof}
    For $r$ sufficiently small, the ball of radius $r$ around $p$ is isometric to a disjoint union of intervals of the form $[0,r]$, identified at the origin, where the number of intervals is precisely the valence of $p$. In the relative part of our filtration, we compute the homology of this space after identifying its boundary, a space homotopy equivalent to a wedge of circles, where the number of circles is now the valence of $p$ minus one. The rank of this homology group is then equal to the valence of $p$ minus one, until $r=0$ and the homology groups vanish. Thus the number of intervals with death time zero in the one-dimensional part of $\Psi_{G}(p)$ is the valence of $p$ minus one. 
    \end{proof}

        \begin{corollary}
    \label{cor:distinctvalence}
    If $p$ and $q$ are two basepoints with different valences, then $\Psi_{G}(p) \neq \Psi_{G}(q)$. In particular, in our setting where there are no vertices of valence two, it is impossible for a vertex and a non-vertex to produce the same persistence diagram.
    \end{corollary}
    
    \begin{lemma}
    \label{lem:upordownfork}
    For any basepoint $p \in G$, and any vertex $v \in G$  not equal to $p$, $v$ is either an upfork, a downfork, or both, in $\Phi_{G}(p)$. If $v$ is a leaf vertex it is necessarily a downfork.
    \end{lemma}
    \begin{proof}
    If $v$ is a leaf, then the edge to which it is adjacent is by necessity the initial segment of a geodesic from $v$ to $p$, making $v$ a downfork. Otherwise, $v$ has valence at least three, so either two or more directions adjacent to $v$ are the initial segments of geodesics from $v$ to $p$, or at most one is. In the former case, $v$ is a downfork. In the latter case, $v$ is an upfork. If $v$ has valence at least four it is possible for it to be both an upfork and a downfork, depending on how many of the adjacent directions to $v$ are the initial segments of geodesics from $v$ to $p$. 
    \end{proof}

\subsection*{Strategy}    
    In light of Corollary \ref{cor:distinctvalence}, we can split our casework into two parts: comparing vertices on the one hand, and comparing non-vertices on the other hand. In either case, our strategy will be the same. We will assume that distinct basepoints produce the same persistence diagram and then deduce the existence of a nontrivial integer linear equality satisfied by the edge lengths of $G$, thus giving a contradiction. It will be apparent from the construction of these equalities that only finitely many integer coefficients show up.

\subsection*{Comparing Vertices}    
We now consider the implications of $\Psi_{G}(v) = \Psi_{G}(w)$ for vertices $w \neq v$. 
    \begin{prop}
    \label{prop:vertexinj}
    If $v,w$ are distinct vertices in $G$, and $\Psi_{G}(v) = \Psi_{G}(w)$, then there is a nontrivial linear equality among edge lengths of $G$.
    \end{prop}
    \begin{proof}
    Suppose first that one of $v$ or $w$ has valence one, so that the other does as well by Lemma \ref{lem:detectvalence}. Then $v$ and $w$ are both leaf vertices, sitting on edges $[v,v']$ and $[w,w']$. By Lemma \ref{lem:upforkvertex}, the smallest nonzero death time in the persistence diagram of $v$ or $w$ is the distance to the closest vertex in $G$ of valence at least three. Since $G$ is connected and does not consist of a single edge, both $v'$ and $w'$ have valence at least three. Thus the smallest nonzero death times in $\Psi_{G}(v)$ and $\Psi_{G}(w)$ are the lengths of $[v,v']$ and $[w,w']$ respectively. $\Psi_{G}(v) = \Psi_{G}(w)$ then implies that these two edges have the same length, a non-trivial linear equality.\\
    
    Suppose next that $v$ and $w$ are distinct vertices of valence three or more. Note that in a connected graph containing at least two distinct vertices of valence at least three, and in which there are no vertices of valence two, every vertex is adjacent to a vertex of valence at least three. Moreover, observe that every path from $v$ or $w$ to a vertex of valence at least three can only hit vertices of valence at least three along the way, as there are no vertices of valence two and a vertex of valence one is a dead-end. Let $v'$ and $w'$ be the closest vertices to $v$ and $w$ respectively with valence at least three. By the prior observation, $v'$ and $w'$ must be adjacent to $v$ and $w$ respectively. This tells us that the smallest nonzero death time in $\Psi_{G}(v)$ or $\Psi_{G}(w)$ is the length of the edge $[v,v']$ or $[w,w']$. The equality $\Psi_{G}(v) = \Psi_{G}(w)$ then implies that these edges have the same length, a nontrivial linear equality unless $v' = w$ and $w' = v$.\\

    If $v' = w$ and $w' = v$, then $v$ and $w$ are the closest vertices of valence at least three to each other, and in fact are the unique closest such vertices, otherwise a non-trivial linear equality among edge lengths has occurred. Let $u$ be the closest vertex to the edge $[v,w]$ among all vertices distinct from $v$ and $w$. If there is more than one such closest vertex, or if $u$ is equidistant from $v$ and $w$, we will have found a nontrivial linear equality among edge lengths. Otherwise, there is a unique closest vertex $u$, and it is strictly closer to one of $v$ or $w$, say $v$. By Lemma \ref{lem:upordownfork}, two possibilities emerge: either $u$ is an upfork (and potentially also a downfork, this is irrelevant) in $\Phi_{G}(v)$, or it is not an upfork, and hence must be a downfork. The latter case, that $u$ is a downfork, implies the existence of distinct geodesics from $u$ to $v$, and hence a nontrivial linear equality among edge lengths. We claim that the former case is impossible: if $u$ is an upfork, the dictionary of \ref{sec:reebdictionary} tells us that the distance $d(u,v)$ is a death time in $\Psi_{G}(v)$. Since $v$ and $w$ are the unique closest vertices to each other, $d(u,v), d(u,w) > d(v,w)$. We have chosen $u$ so that $d(v,w) < d(u,v) < d(x,w)$ for any vertex $x \neq u,v$ of valence at least three. Thus, by Lemma \ref{lem:upforkvertex}, which tells us that death times in $\Psi_{G}(w)$ correspond to distances from $w$ to vertices $x$ of valence at least three, we see that there is no such vertex $x$ for which $d(w,x) = d(u,v)$, and hence the death time $d(u,v)$ in $\Psi_{G}(v)$ cannot be matched by anything in $\Psi_{G}(w)$, so their barcodes cannot be equal.
    
    \end{proof}

\subsection*{Comparing Non-Vertex Points.}
    For the remainder of the proof, we will want to show the analogous result for non-vertex points. It will be useful to distinguish three kinds of basepoints $p \in G$:
    \begin{itemize}
    	\item Case A: The point $p$ sits on a leaf edge. See Figure \ref{CaseA}
    	\item Case B: The point $p$ sits on a non-leaf edge, with both endpoint vertices being upforks in the Reeb graph $\Phi_{G}(p)$. See Figure \ref{CaseB}
    	\item Case C: The point $p$ sits on a non-leaf edge, and exactly one of the endpoint vertices is an upfork in the Reeb Graph $\Phi_{G}(p)$. To be precise, the closer vertex $v$ is an upfork and the further one $w$ is a downfork. This means that there are at least two geodesics from $w$ to $p$ starting along different edges adjacent to $w$. We claim that, without loss of generality, one of these geodesics is the subsegment of $E_p$ from $w$ to $p$, as illustrated in Figure \ref{CaseC}. If not, both geodesics pass through $v$, and we find that there are two distinct geodesics from $w$ to $v$. The existence of distinct geodesics between vertices implies that the edge lengths of $G$ satisfy a nontrivial linear equality (indeed, the sum of edge lengths along one geodesic equals the sum of those along the other), which immediately gives the claimed contradiction.
    \end{itemize}

This case analysis is exhaustive because (1) Lemma \ref{lem:upordownfork} implies that all vertices must be upforks and/or downforks, and (2) the last remaining possibility, namely that $p$ sits on a non-leaf edge and both endpoints are downforks, implies that $p$ sits on a self-loop, and we have assumed that $G$ contains no self-loops.

    \begin{figure}[htb]
    	\labellist
    	\small\hair 2pt
    	\pinlabel $p$ at 50 22
    	\pinlabel $v$ at 68 44
    	\pinlabel $E_p$ at 26 53
    	\endlabellist
        \centering
        \includegraphics{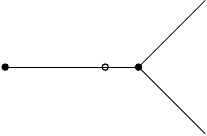}
        \caption{Case A}
        \label{CaseA}
    \end{figure}
    
      \begin{figure}[htb]
      	    	\labellist
      	\small\hair 2pt
      	\pinlabel $p$ at 50 -10
      	\pinlabel $w$ at 25 63 
      	\pinlabel $v$ at 70 31
      	\pinlabel $E_p$ at 46 32
      	\endlabellist
      	\centering
        \includegraphics{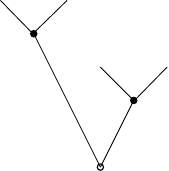}
        \caption{
        Case B. Note that $p$ is not a vertex, but the edge $E_p$ is drawn as it appears in the Reeb graph $\Phi_{G}(p)$, with height proportional to the distance from $p$.} 
        \label{CaseB}
    \end{figure}

      \begin{figure}[htb]
      	      	    	\labellist
      	\small\hair 2pt
      	\pinlabel $p$ at 0 -5
      	\pinlabel $w$ at 0 57 
      	\pinlabel $v$ at 41 14
      	\pinlabel $E_p$ at 15 22
      	\endlabellist
        \centering
        \includegraphics[scale =1.5]{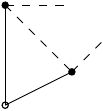}
        \caption{Case C}
        \label{CaseC}
    \end{figure}

Before we continue our proof, we note that there are two very simple pieces of geometric data that can be read off of a persistence diagram.

\begin{lemma}
\label{lem:d1}
For $p \in G$ not a vertex, the smallest nonzero death time is the distance from $p$ to the closest vertex of valence at least three.
\end{lemma}
\begin{proof}
	By Lemma \ref{lem:upforkvertex}, death times correspond to distances from $p$ to certain vertices of valence at least three. The closest such vertex is necessarily an upfork in $\Phi_{G}(p)$, giving rise to the smallest death time, as if it were an upfork it would be the base of a self-loop containing $p$, impossible by hypothesis.
\end{proof}
\begin{remark}
\label{rem:diam}
Secondly, for any point $p \in G$, the zero-dimensional part of the persistence diagram $\Psi_{G}(p)$ contains a single point $(0,D)$, where $D$ is the radius of the metric space at $p$ -- the furthest distance from $p$ to another point in $G$.
\end{remark}

\begin{definition}
For an edge $E$ in a metric graph $(G,d_G)$, $L(E)$ will denote the length of $E$.
\end{definition}

Next, the following lemma will be useful in the case analysis to come.

\begin{lemma}
\label{lem:sameedge}
Let $p,q$ be two non-vertex points in $G$ sitting on the same non-boundary edge $E$. If $\Psi_{G}(p) = \Psi_{G}(q)$ then there is a nontrivial linear equality among edge lengths of $G$.
\end{lemma}
\begin{proof}
Consider Figure \ref{fig:genericproofsameedge}. The closest vertex of valence at least three to $p$ is $v$, as since $d(q,w) < d(p,w)$ we could otherwise deduce that $\Psi_{G}(p) \neq \Psi_{G}(q)$ from Lemma \ref{lem:d1}. Similarly, the closest vertex of valence at least three to $q$ is $w$. From Lemma \ref{lem:d1}, we know that $d_1 = d(p,v) = d(q,w)$. Note that $p$ must be strictly closer to $v$ than $q$, and $q$ must be strictly closer to $w$ than $p$, for if, say, there is a geodesic from $q$ to $v$ of length $\leq d(p,v)$, then this geodesic passes through $w$, meaning that $d_1 = d(q,w) < d(p,v) = d_1$, a contradiction.\\
    
Take $u$ to be a closest vertex to either $v$ or $w$ among the other vertices in $G$, and suppose without loss of generality that it is closer to $v$. If there is more than one choice for $u$, or if it is equidistant from $v$ and $w$, then we have a nontrivial linear equality among edge lengths, so otherwise we may assume that $d(u,v) < d(u,w)$ and $d(u,v) < d(x,v),d(x,w)$ for any fourth vertex $x$ of valence at least three. If $u$ is an upfork { in the Reeb graph $\Phi_{G}(p)$}, then by Lemma \ref{lem:upforkvertex} it gives rise to a death time in {$\Psi_{G}(p)$} strictly greater than $d_1$, and strictly smaller than any other death time in {$\Psi_{G}(q)$, so that $\Psi_{G}(p) \neq \Psi_{G}(q)$}, contrary to our hypothesis. Thus $u$ is a downfork, so that there are two geodesics from $u$ to $p$, and either both geodesics pass through $v$ or one passes through $w$. It is impossible for one to pass through $w$, as any path from $u$ to $p$ passing through $w$ has length strictly longer than $L(E) + d_1$, and hence cannot be a geodesic. Thus both pass through $v$, and hence there are distinct geodesics from from $u$ to $v$, implying some nontrivial linear equality among edge lengths.
    \end{proof}

    \begin{figure}[htb]
    	\labellist
    	\small\hair 2pt
    	\pinlabel $v$ at 19 17
    	\pinlabel $u$ at 8 86
    	\pinlabel $E$ at 22 61
    	\pinlabel $p$ at 49 23
    	\pinlabel $d_1$ at 34 44
    	\pinlabel $q$ at 116 20
    	\pinlabel $d_1$ at 129 44
    	\pinlabel $w$ at 146 17
    	\endlabellist
	\centering
	\includegraphics{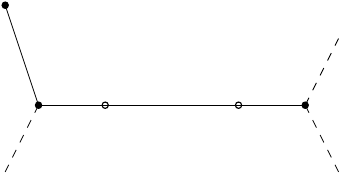}
	\caption{
		When $E = E_p = E_q$. Note that $\Psi_{G}(p) = \Psi_{G}(q)$ and Lemma \ref{lem:d1} imply that $p,q$ are equidistant to their closest vertices of valence at least three, which are $v$ and $w$ respectively.}
	
	\label{fig:genericproofsameedge}
\end{figure}
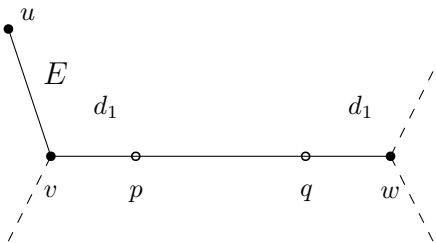

Finally, we prove a lemma that will prove useful in both this section and section \ref{sec: selfloops}.

\begin{lemma}
	\label{lem:deathzero}
	Let $G$ be any metric graph, $p \in G$ a basepoint, and $(b,\cdot)$ a point in the one-dimensional persistence diagram of $\Psi_{G}(v)$. If $p$ is a vertex, then $2b$ is either (i) the length of a simple cycle in $G$ containing $p$, or (ii) the sum of lengths of a collection of edges $\sum_{i \in U}L(E_i)$. If $p$ is a non-vertex, let $d_1$ and $d_2$ be the distances from $p$ to the vertices at the ends of its edge $E_p$. Then either (i) $2b$ is the length of a simple cycle in $G$ containing $p$, or (ii) $2(b-d_1)$ is the sum of lengths of a collection of edges $\sum_{i \in U}L(E_i)$ omitting $E_p$, or (iii) $2(b-d_2)$ is the sum of lengths of a collection of edges $\sum_{i \in U}L(E_i)$ omitting $E_p$.
\end{lemma}
\begin{proof}
	The point $(b,\cdot)$ corresponds to a downfork $q$ in the Reeb graph $\Phi_{G}(p)$ at distance $b$ from $p$. The downfork $q$ may either be a leaf vertex or a point (not necessarily a vertex) of valence at least two.\\
	
	Suppose first that $q$ is a leaf vertex. If $p$ is a vertex then $b$ is the sum of the lengths of the edges along this geodesic, and hence so is $2b$ (by repeating each edge twice): see the left-hand side of Figure \ref{fig:deathzerofig1}. If $p$ is not a vertex, then the geodesic from $q$ to $p$ passes through one of the vertices on the boundary of $E_p$, and so either $b-d_1$ or $b-d_2$ is the sum of the length of the edges along the geodesic from $q$ to this vertex. Thus either $2(b-d_1)$ or $2(b-d_2)$ is also the sum of the length of edges: see the right-hand side of Figure \ref{fig:deathzerofig1}. Note that when $p$ is a non-vertex, the edge $E_p$ does not appear in the edges in the sum.\\
	
	Suppose next that $q$ is not a leaf vertex. As $q$ is a downfork, there are at least two geodesics from $q$ to $p$. These geodesics either (i) first meet at $p$, (ii) meet before arriving at $p$. In scenario (i), the sum of the two geodesics is a simple cycle of length $2b$. See Figure \ref{fig:deathzerofig3}.\\
	
	In scenario (ii), if $p$ is a vertex, then $2b$ is the sum of edge lengths as show on the left-hand side of Figure \ref{fig:deathzerofig2}. If $p$ is a non-vertex, then either $2(b-d_1)$ or $2(b-d_2)$ is the sum of edge lengths, as shown on the right-hand side of Figure \ref{fig:deathzerofig2}. Note that when $p$ is a non-vertex, the edge $E_p$ does not appear in the edges in the sum.

\end{proof}
	\begin{figure}[htb]
		    	\labellist
		\small\hair 2pt
		\pinlabel $p$ at 40 29
		\pinlabel $q$ at 8 135
		\pinlabel $p$ at 155 15
		\pinlabel $d_1$ at 150 35
		\pinlabel $q$ at 138 135
		\endlabellist
	\centering
	\includegraphics[]{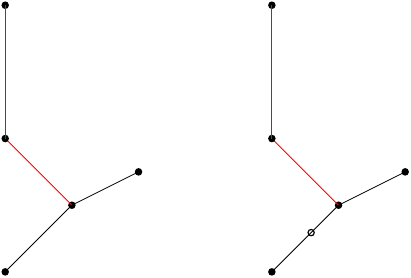}
	\caption{The edges appearing in the sum are drawn in red.}
	\label{fig:deathzerofig1}
\end{figure}
	\begin{figure}[htb]
		    	\labellist
		\small\hair 2pt
		\pinlabel $p$ at 72 1
		\pinlabel $q$ at 55 101 
		\endlabellist
	\centering
	\includegraphics[]{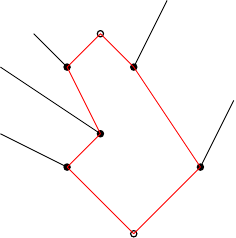}
	\caption{The simple cycle of of length $2b$ is drawn in red.}
	\label{fig:deathzerofig3}
\end{figure}
	\begin{figure}[htb]
		    	\labellist
		\small\hair 2pt
		\pinlabel $p$ at 65 -5
		\pinlabel $q$ at 86 119 
		\pinlabel $q$ at 245 119 
		\pinlabel $p$ at 200 17
		\pinlabel $d_2$ at 215 22
		\endlabellist
	\centering
	\includegraphics[]{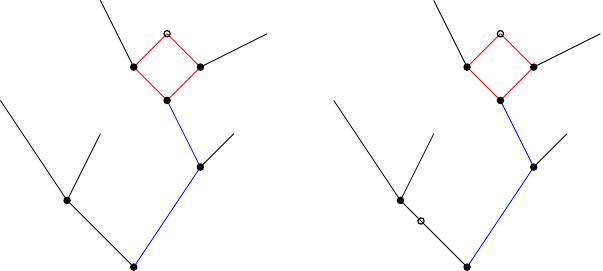}
	\caption{The edges appearing in the sum are drawn in red or blue. Red edges appear once in the sum and blue edges appear twice.}
	\label{fig:deathzerofig2}
\end{figure}

We now come to the proof itself. The following two propositions demonstrate the result of Proposition \ref{prop:vertexinj} for non-vertices. The first proposition deals with pairs of basepoints in distinct cases (among the cases A,B, and C, as defined above), and the second proposition deals with pairs of basepoints of the same case. It is important to note that our persistence diagrams come labelled by dimension but do not tell us if a point comes from ordinary, relative, or extended persistence. Some of the following casework emerges as a result of this ambiguity.{ 
\begin{prop}
\label{prop:nonvertexdiffcasesinj}
    If $v,w$ are distinct non-vertex points in $G$ of distinct cases, and $\Psi_{G}(v) = \Psi_{G}(w)$, then there is a nontrivial linear equality among edge lengths of $G$.
    \end{prop}
    
    \begin{proof}
    Our proof consists of three parts, pertaining to which pairs of cases our points belong: (1) cases A and B, (2) cases B and C, and (3) cases A and C. In all these cases, $\Psi_{G}(v) = \Psi_{G}(w)$ and Lemma \ref{lem:d1} implies that the distances from $v$ and $w$ to their closest vertices of valence at least three are equal, and denoted $d_1$.\\

    {\bf Cases A and B:}
    \begin{figure}[htb]
    	    	\labellist
    	\small\hair 2pt
    	\pinlabel $E_p$ at 29 20
    	\pinlabel $p$ at 50 23 
    	\pinlabel $d_2$ at 18 42
    	\pinlabel $d_1$ at 57 42
    	\pinlabel $q$ at 164 4
		\pinlabel $E_q$ at 164 41
		\pinlabel $r_2$ at 141 41
		\pinlabel $d_1$ at 180 28
    	\pinlabel $w$ at 141 77
\pinlabel $v$ at 187 42

    	\endlabellist
        \centering
        \includegraphics{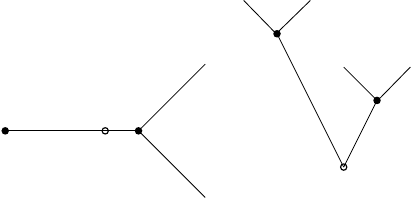}
        \caption{Cases A and B}
        \label{fig:AB}
    \end{figure}
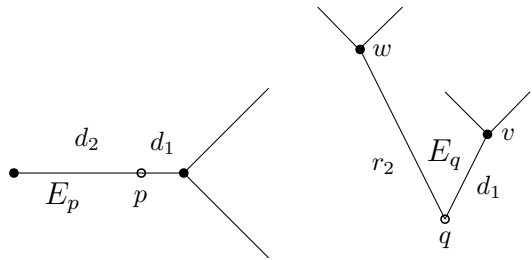
    
    Suppose $p$ is of case $A$ and $q$ is of case $B$, with the edge lengths as shown in Figure \ref{fig:AB} We know that
    
    \begin{equation}
    \label{eqn:gen1}
       d_1 + d_2 = L(E_p) 
    \end{equation}
\begin{equation}
\label{eqn:gen2}
     d_1 + r_2 = L(E_q)\
\end{equation}

 Clearly, $E_p \neq E_q$ as one is a leaf edge and the other is not.\\
    
    By hypothesis, $r_2$ is a death time in $\Psi_{G}(q)$, corresponding in $\Psi_{G}(p)$ to the distance from $p$ to a vertex of valence at least three. Since the geodesic from $p$ to any such vertex passes through the segment of $E_p$ of length $d_1$, we can deduce that $r_2 - d_1$ is the sum of the lengths of edges in $G$, where none of these edges is $E_p$ by construction, and none of them are $E_q$ either, as $r_2 - d_1 < r_2 + d_1 = L(E_q)$. Thus there is a set of edges $U$ along a geodesic between vertices in $G$, omitting $E_p$ and $E_q$, for which 
    
    \begin{equation}
    \label{eqn:gen3}
    r_2 - d_1 = \sum_{i \in U} L(E_i)    
    \end{equation}

     Lastly, $\Psi_{G}(p)$ contains the point {$(d_2,0)$}. This cannot be a point in zero-dimensional persistence, i.e. that the furthest point from $p$ is the leaf vertex at distance $d_2$. Indeed, the point $q$ is necessarily further away than $p$ from this leaf vertex, and hence contains a larger zero-dimensional death time, making it impossible for $\Psi_{G}(p) = \Psi_{G}(q)$. Thus {$(d_2,0)$} is a point in one-dimensional persistence for $\Psi_{G}(p)$, and hence, correspondingly for $\Psi_{G}(q)$. {By Lemma \ref{lem:deathzero}, three possibilities emerge: either $2d_2$ is the length of a simple cycle in $G$, or $2(d_2 - d_1)$ is the sum of edges in $G$, omitting $E_q$, or $2(d_2 - r_2)$ is the sum of edges in $G$, omitting $E_q$.\\
      
    In the first case, taking 2(\ref{eqn:gen1}) + (\ref{eqn:gen3})  - (\ref{eqn:gen2}) gives
    \[2d_2 = 2L(E_p) + \sum_{i \in U} L(E_i) - L(E_q)\]
    
   Now, this should equal the length of a simple cycle in $G$. However, $E_p$ cannot appear among the edges of this cycle, as it is a leaf edge. This implies a nontrivial linear equality among edge lengths.\\
   
   In the second case, we have that
     \begin{equation}
     \label{eqn:gen4}
    d_2 - d_1 = \frac{1}{2} \sum_{i \in S} L(E_i)     
     \end{equation}
     where this sum of edges omits $E_q$. Taking (\ref{eqn:gen2}) - (\ref{eqn:gen1}) - (\ref{eqn:gen3}) + (\ref{eqn:gen4}), we obtain
    \[0 = L(E_q) -L(E_p) + \frac{1}{2}\sum_{i \in S} L(E_i) - \sum_{i \in U} L(E_i)\]
    
    Multiplying both sides by two gives a nontrivial linear equality among edge lengths. Indeed, the right side cannot cancel out because neither collection $U$ nor $S$ contain $E_q$.\\

     In the third case, we have that
     \begin{equation}
     \label{eqn:gen5}
    d_2 - r_2 = \frac{1}{2} \sum_{i \in S} L(E_i)     
     \end{equation}
     and again this sum of edges omits $E_q$. Taking (\ref{eqn:gen2}) - (\ref{eqn:gen1}) + (\ref{eqn:gen5}) gives
    \[0 = L(E_q) - L(E_p) + \frac{1}{2}\sum_{i \in S} L(E_i)\]
    
    multiplying by two, as before, produces a nontrivial linear equality among edge lengths.}\\

     {\bf Cases B and C:}

    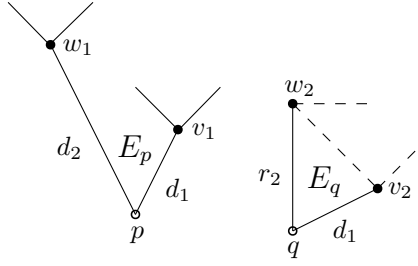
\begin{figure}[htb]
    	    	\labellist
\small\hair 2pt
\pinlabel $E_p$ at 47 34
\pinlabel $p$ at 48 0
\pinlabel $d_2$ at 26 32
\pinlabel $d_1$ at 62 19
\pinlabel $v_1$ at 71 36
\pinlabel $w_1$ at 26 72
\pinlabel $q$ at 108 -5
\pinlabel $E_q$ at 119 20
\pinlabel $r_2$ at 100 20
\pinlabel $d_1$ at 126 3
\pinlabel $w_2$ at 110 56
\pinlabel $v_2$ at 148 17

\endlabellist
        \centering
        \includegraphics{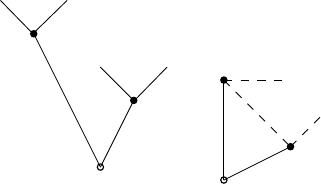}
        \caption{Cases B and C}
        \label{fig:BC}
    \end{figure}
    
    Refer to Figure \ref{fig:BC}. In this case, we see that
    
    \begin{equation}
    \label{eqn:gen6}
      d_1 + d_2 = L(E_p)  
    \end{equation}
    \begin{equation}
    \label{eqn:gen7}
       d_1 + r_2 = L(E_q)\ 
    \end{equation}

    To start, we assume that $E_p \neq E_q$ by passing to Lemma \ref{lem:sameedge}. As before, $d_2$ is a death time in $\Psi_{G}(p)$, corresponding in $\Psi_{G}(q)$ to the distance from $q$ to a vertex of valence three or greater. Without loss of generality this geodesic passes through $v_2$, as by the hypotheses of case C any geodesic passing through $w_2$ can be re-routed to pass through $v_2$ and have the same length. We see that 
    \begin{equation}
    \label{eqn:gen8}
    d_2 - d_1 = \sum_{i \in U} L(E_i)    
    \end{equation}
     where this collection of edges omits $E_p$ and $E_q$. We can also see that the point {$(r_2,0)$} shows up in the one-dimensional persistence of $\Psi_{G}(q)$, and hence also in $\Psi_{G}(p)$. As before, we end up with three possibilities.{ If there is a  simple cycle in $G$ of length $2r_2$, then taking} 2(\ref{eqn:gen7}) + (\ref{eqn:gen8}) - (\ref{eqn:gen6}) gives
    \[2r_2 = 2L(E_q)  + \sum_{i \in U} L(E_i) - L(E_p)\]
    This should equal the length of a simple cycle in $G$. However, in such a cycle each edge shows up once, whereas the right-hand side above has $E_p$ appear twice, so that we must have a nontrivial linear equality among edge lengths.\\

    {Otherwise, Lemma \ref{lem:deathzero} guarantees that either
        \begin{equation}
        \label{eqn:gen9}
    r_2 - d_1 = \frac{1}{2} \sum_{i \in S} L(E_i)     
     \end{equation}
     or
        \begin{equation}
        \label{eqn:gen10}
    r_2 - d_2 = \frac{1}{2} \sum_{i \in S} L(E_i)     
     \end{equation}
    In either case, the sum of edges omits $E_q$, and the resulting analysis is the same as when comparing cases A and B.}\\

     {\bf Cases A and C:}

    \begin{figure}[htb]
    	    	\labellist
\small\hair 2pt
\pinlabel $p$ at 50 24
\pinlabel $d_2$ at 22 22
\pinlabel $d_1$ at 60 22
\pinlabel $q$ at 130 -5
\pinlabel $E_q$ at 141 25
\pinlabel $r_2$ at 120 28
\pinlabel $d_1$ at 149 9
\pinlabel $w$ at 128 60
\pinlabel $v$ at 167 18
\endlabellist
        \centering
        \includegraphics{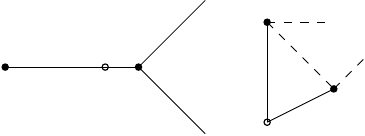}
        \caption{Cases A and C}
        \label{fig:AC}
    \end{figure}
    
    Refer to Figure \ref{fig:AC}. The one-dimensional persistence in $\Psi_{G}(q)$ contains the point {$(r_2,0)$}. As for the point {$(d_2,0)$} in $\Psi_{G}(p)$, it may either be a point in one-dimensional or zero-dimensional persistence. Lemma \ref{lem:detectvalence} tells us that $p$ and $q$ both have a single point of the form {$(\cdot,0)$} in one-dimensional persistence. If {$(d_2,0)$} is a point in one-dimensional persistence, $\Psi_{G}(p) = \Psi_{G}(q)$ implies $r_2 = d_2$ and thus $L(E_p) = L(E_q)$, and so we obtain a nontrivial equality among edge lengths. In the latter case, the radius at $p$ is $d_2$, realized by the distance from $p$ to its adjacent leaf vertex; since $q$ is strictly further from this leaf vertex than $p$, it has a larger radius, which violates our assumption that $\Psi_{G}(p) = \Psi_{G}(q)$.
    \end{proof}
    
    \begin{prop}
    \label{prop:nonvertexsamecaseinj}
    If $v,w$ are distinct non-vertex points in $G$ of the same kind, and $\Psi_{G}(v) = \Psi_{G}(w)$, then there is a nontrivial linear equality among edge lengths of $G$.
    \end{prop}
    
    \begin{proof}

        {\bf Cases A and A:}\\

    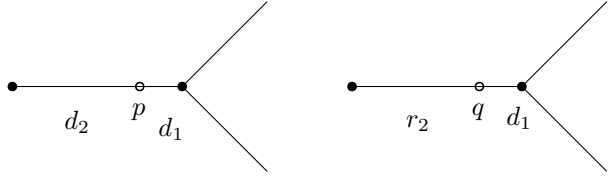
\begin{figure}[htb]
    	    	\labellist
\small\hair 2pt
\pinlabel $p$ at 51 24
\pinlabel $d_2$ at 22 25
\pinlabel $d_1$ at 62 25
\pinlabel $q$ at 179 23
\pinlabel $r_2$ at 150 25
\pinlabel $d_1$ at 189 25
\endlabellist
        \centering
        \includegraphics{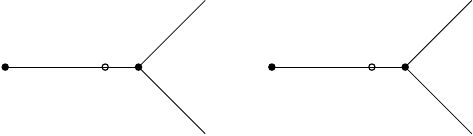}
        \caption{Case A and Case A}
        \label{fig:AA}
    \end{figure}

    Suppose $p$ and $q$ are both of case A, as in Figure \ref{fig:AA}. Note that for every $d \in [0,L(E_p)]$, the edge $E_p$ contains a unique point $x$ at distance $d$ from its unique vertex of valence at least three, which will necessarily be the smallest nonzero death time in $\Psi_{G}(x)$ by Lemma \ref{lem:d1}, and similarly for $E_q$. If $E = E_p = E_q$ then, again by Lemma \ref{lem:d1}, both $p$ and $q$ would be the same distance $d$ from the non-leaf vertex of $E$, implying $p=q$. Thus, if $p \neq q$, we may deduce $E_p \neq E_q$.\\  
    
   Moving on, $\Psi_{G}(p)$ contains the point {$(d_2,0)$} and $\Psi_{G}(q)$ contains the point {$(r_2,0)$}. We claim both of these points are in one-dimensional persistence, for if, say, {$(d_2,0)$} is a point in zero-dimensional persistence, the radius at $p$ is $d_2$, realized by the distance from $p$ to its adjacent leaf vertex. As we observed earlier, since $q$ is strictly further from this leaf vertex than $p$, it has a larger radius, which violates our assumption that $\Psi_{G}(p) = \Psi_{G}(q)$.\\
    
  Lemma \ref{lem:detectvalence} tells us that $p$ and $q$ both have a single point of the form {$(\cdot,0)$} in one-dimensional persistence. Thus if both {$(d_2,0)$} and {$(r_2,0)$} are points in one-dimensional persistence, we must have $d_2 = r_2$, and hence $L(E_p) = L(E_q)$, a nontrivial equality among edge lengths since $E_p \neq E_q$.\\

        {\bf Cases B and B:}\\

    \begin{figure}[htb]
    	    	\labellist
    	\small\hair 2pt
    	\pinlabel $p$ at 53 -5
    	\pinlabel $d_1$ at 60 12 
    	\pinlabel $d_2$ at 27 26
    	\pinlabel $E_p$ at 48 27
    	\pinlabel $q$ at 131 -5
		\pinlabel $d_1$ at 140 12 
		\pinlabel $r_2$ at 108 26
		\pinlabel $E_q$ at 128 27
    	\endlabellist
        \centering
        \includegraphics{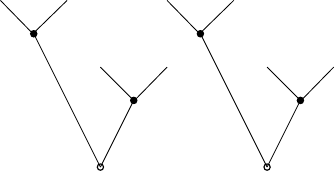}
        \caption{Case B and Case B}
        \label{fig:BB}
    \end{figure}
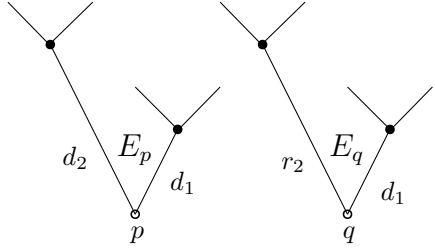
    
   Refer to figure \ref{fig:BB}. We observe that
   \begin{equation}
   \label{eqn:gen11}
    d_1 + d_2 = L(E_p)   
   \end{equation}
   \begin{equation}
   \label{eqn:gen12}
    d_1 + r_2 = L(E_q)   
   \end{equation}

  Using Lemma \ref{lem:sameedge}, we assume without loss of generality that $E_p \neq E_q$. If $d_2 = r_2$ then $L(E_p) = L(E_q)$, a nontrivial linear equality. Thus let us assume, again without loss of generality, that $d_2 < r_2$. Since $d_2$ is a death time for $p$, and hence corresponds to the distance from $q$ to a vertex, and since the corresponding geodesic must pass along the segment of length $d_1$ adjacent to $q$ due to length considerations, we deduce further that
   \begin{equation}
   \label{eqn:gen13}
    d_2 - d_1 = \sum_{i \in S} L(E_i)   
   \end{equation}
    Since $d_2 - d_1 < L(E_p)$ and $d_2 - d_1 < d_2 < r_2 < L(E_q)$, none of the edges present in the sum can be $E_p$ or $E_q$.\\
   
   On the other hand, since $r_2$ is a death time for $q$, this must correspond to a geodesic from $p$ to a vertex in $G$, passing either (i) through the segment of length $d_1$ or (ii) through the segment of length $d_2$. In (i), we have that
   \begin{equation}
   \label{eqn:gen14}
    r_2 - d_1 = \sum_{i \in T} L(E_i)   
   \end{equation}
   By construction, none of the edges in this sum are $E_p$, as the geodesic of length $r_2$ starts at the non-vertex point $p$ on $E_p$, and subtracting away $d_1$ removes the length of the segment from $p$ to a boundary vertex of $E_p$. Moreover, $r_2 - d_1 < L(E_q)$, so this sum of edges omits $E_q$ as well. Taking (\ref{eqn:gen11}) - (\ref{eqn:gen12}) - (\ref{eqn:gen13}) + (\ref{eqn:gen14}) gives
    
    \[0 = L(E_p) - L(E_q) - \sum_{i \in S} L(E_i) + \sum_{i \in T} L(E_i)\]
    
    a nontrivial linear equality among edge lengths.\\
   
   In (ii), we have 
   \begin{equation}
   \label{eqn:gen15}
    r_2 - d_2 = \sum_{i \in T} L(E_i)   
   \end{equation}
   the sum of edge lengths not including $E_p$ or $E_q$ (by the same argument as in the prior case). Taking (\ref{eqn:gen15}) + (\ref{eqn:gen11}) - (\ref{eqn:gen12}) gives

    \[0 = \sum_{i \in T} L(E_i) + L(E_p) - L(E_q)\]
    
    again a nontrivial linear equality among edge lengths.\\
    
          {\bf Cases C and C:}\\
    \begin{figure}[htb]
    	    	\labellist
    	\small\hair 2pt
    	\pinlabel $p$ at 0 -5
    	\pinlabel $d_1$ at 20 5 
    	\pinlabel $d_2$ at -6 26
    	\pinlabel $E_p$ at 14 20
    	 \pinlabel $q$ at 68 -5
    	\pinlabel $d_1$ at 84 5
    	\pinlabel $r_2$ at 60 26
    	\pinlabel $E_q$ at 78 20
    	\endlabellist
        \centering
        \includegraphics{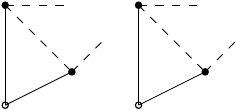}
        \caption{Case C and Case C}
        \label{fig:CC}
    \end{figure}
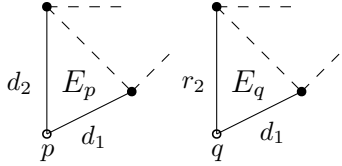
    Refer to Figure \ref{fig:CC}. Let us assume, using Lemma \ref{lem:sameedge}, that $E_p \neq E_q$. We know that {$(d_2,0)$} shows up in the 1-dimensional persistence of $\Psi_{G}(p)$ and {$(r_2,0)$} shows up in the 1-dimensional persistence of $\Psi_{G}(q)$. By Lemma \ref{lem:detectvalence}, both $p$ and $q$ have a single point of the form {$(\cdot,0)$} in one dimensional persistence. Thus, if $\Psi_{G}(p) = \Psi_{G}(q)$, then we must have $d_2 = r_2$ and hence $L(E_p) = L(E_q)$. Since $E_p \neq E_q$, this is a nontrivial equality among edge lengths.\\
     
\end{proof}

Proposition \ref{genericinj} now follows from Propositions \ref{prop:vertexinj}, \ref{prop:nonvertexdiffcasesinj}, \ref{prop:nonvertexsamecaseinj}, and Corollary \ref{cor:distinctvalence}.

\section{The Case of Self-Loops and Few Vertices of Valence not Equal to Two}
 \label{sec: selfloops}
 
We have noted that certain combinatorial types, such as combinatorial graphs $X$ with topological self-loops, force automorphisms regardless of the geometry chosen. Thus we cannot hope for $\Psi_{X}$ to be injective for any choice of edge lengths on $X$. Still, we claim that it is generically possible, in these cases, to reconstruct a metric graph $G$ from $\mathcal{BT}(G)$. As in the prior section, we remove all vertices of valence two, merging their adjacent edges into a single, longer edge.\\

We now demonstrate that the failure of injectivity introduced by self-loops is generically isolated to the self-loop, and that $\Psi_{G}$ does not identify a point on a self-loop with a point outside it. 

\begin{lemma}
\label{lem:loopnonloop}
Let $G$ be a metric graph with a topological self-loop $C$. Suppose that there exist pairs of points $p, q \in G$, with $p \in C$ and $q \notin C$, such that $\Psi_{G}(p) = \Psi_{G}(q)$. Then there is a nontrivial integer linear equality among the edge lengths of $G$.
\end{lemma}
\begin{proof}
By Lemma \ref{lem:detectvalence}, $p$ and $q$ are either both vertices or both non-vertices. Suppose first that they are vertices. $\Psi_{G}(p) = \Psi_{G}(q)$ implies that the latter contains the point $(L(C)/2,0)$. By Lemma \ref{lem:deathzero}, either (i) $L(C)$ is equal to the length of a simple cycle in $G$ containing $q$, or (ii) $L(C) = \sum_{i \in U}L(E_i)$ for some set of edges in $G$. Case (i) implies a nontrivial linear equality among edge lengths, as $q$ does not sit on $C$, and case (ii) implies a nontrivial linear equality unless the set of edges $U$ consists exactly of $C$, impossible if $q \notin C$.\\

Next, let us suppose that $p$ and $q$ are non-vertices. By Lemma \ref{lem:d1}, they are the same distance $d_1$ from their closest vertices of valence at least three.{ The point $p$ sits on a self-loop, but the point $q$ may not. Indeed, there are four possibilities for the point $q$, corresponding to the three cases A, B, and C of section \ref{sec: geninj}, and the fourth case D, when $q$ sits on a self-loop.}\\

{\bf {Case A:}}
\begin{figure}[htb]
	    	\labellist
	\small\hair 2pt
	\pinlabel $p$ at 33 -5
	\pinlabel $x$ at 8 17 
	\pinlabel $E_p$ at 24 29
	\pinlabel $u$ at 41 46
	\pinlabel $d_1$ at 88 26
	\pinlabel $q$ at 192 26 
	\pinlabel $r_2$ at 146 26
	\pinlabel $d_1$ at 230 26
		\pinlabel $E_q$ at 169 55
	\pinlabel $v$ at 255 44
	\endlabellist
    \centering
    \includegraphics{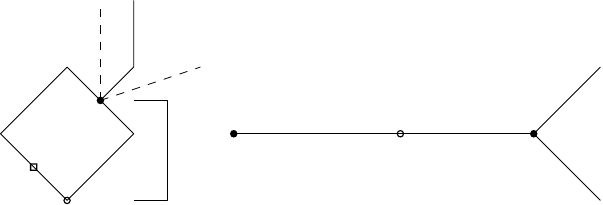}
    \caption{}
    \label{fig:AD}
\end{figure}

Refer to Figure \ref{fig:AD}. The point $q$ sits on a leaf edge $E_q$. Consider the point $x$ that is on the other side of the loop $E_p$ from $u$. This produces a downfork in $\Phi_{G}(q)$, and hence a point in $\Psi_{G}(q)$ with birth time $d_1 + d(v,u) + L(E_p)/2$. $\Psi_{G}(p) = \Psi_{G}(q)$ means that there is a similar birth time in $\Psi_{G}(p)$. This is strictly larger than $L(E_p)/2$, and hence the geodesic to the corresponding downfork passes through $u$.{ By Lemma \ref{lem:deathzero}}, this is equal to $d_1 + \frac{1}{2}\sum_{i \in U} L(E_i)$ for some collection of edges in $G$ omitting $E_p$. This implies $d(v,u) + L(E_p)/2 = \frac{1}{2} \sum_{i \in U} L(E_i)$, a nontrivial linear equality among edge lengths.\\

{\bf {Case B}:}
\begin{figure}[htb]
	    	\labellist
	\small\hair 2pt
	\pinlabel $p$ at 33 -5
\pinlabel $x$ at 8 17 
\pinlabel $E_p$ at 24 29
\pinlabel $u$ at 41 46
\pinlabel $d_1$ at 88 26
\pinlabel $q$ at 160 -5 
\pinlabel $v$ at 182 47
\pinlabel $w$ at 136 65
\pinlabel $d_1$ at 174 23
\pinlabel $r_2$ at 137 29
\pinlabel $E_q$ at 155 40
	\endlabellist
    \centering
    \includegraphics{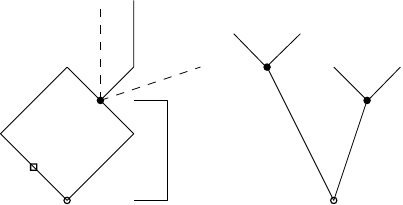}
    \caption{}
    \label{fig:BD}
\end{figure}

{The point $q$ sits on an internal edge $E_q$ whose endpoints are both upforks in the Reeb graph $\Phi_{G}(q)$, as depicted in Figure \ref{fig:BD}}. We have that
\begin{equation}
    \label{eqn:bd1}
    r_2 + d_1 = L(E_q)
\end{equation}
 Moreover, since $r_2$ shows up as a death time in $\Psi_{G}(q)$, it must also be a death time in $\Psi_{G}(p)$, corresponding by the dictionary of Section \ref{sec:reebdictionary} to the distance from $p$ to an upfork in $\Phi_{G}(p)$. The geodesic realizing this distance passes through $u$, and hence
\begin{equation}
    \label{eqn:bd2}
    r_2 - d_1 = \sum_{i \in U}L(E_i)
\end{equation}
where the collection $U$ of edges omits $E_p$ by construction and has a total length too short to contain $E_q$.\\

The one-dimensional persistence diagram in $\Psi_{G}(p)$ contains the point $(L(E_p)/2,0)$.{ By Lemma \ref{lem:deathzero}, we have a trichotomy. The first possibility is that $L(E_p)/2$ is half the length of a simple cycle on which $q$ sits. This implies a nontrivial linear equality as this simple cycle cannot consist of $E_p$. The second possibility is that
\begin{equation}
    \label{eqn:bd3}
    L(E_p)/2 - d_1 = \frac{1}{2}\sum_{i \in S} L(E_i)
\end{equation}
with the sum of edges on the right side omitting $E_q$. Taking (\ref{eqn:bd1}) - (\ref{eqn:bd2}) + 2(\ref{eqn:bd3}) gives
\[L(E_p) =  L(E_q) +\sum_{i \in S}L(E_i) - \sum_{i \in U}L(E_i)\]
a nontrivial linear equality. In the third possibility, we have
\begin{equation}
    \label{eqn:bd4}
    L(E_p)/2 - r_2 = \frac{1}{2}\sum_{i \in S} L(E_i)
\end{equation}
Note that this set $S$ is distinct from the one in (\ref{eqn:bd3}). Taking (\ref{eqn:bd1}) + (\ref{eqn:bd2}) + 2(\ref{eqn:bd4}) gives
\[L(E_p) = L(E_q) + \sum_{i \in U}L(E_i) + \sum_{i \in S}L(E_i)\] 
a nontrivial linear equality.}\\

{\bf {Case C:}}
\begin{figure}[htb]
	    	\labellist
	\small\hair 2pt
		\pinlabel $p$ at 33 -5
	\pinlabel $x$ at 8 17 
	\pinlabel $E_p$ at 24 29
	\pinlabel $u$ at 41 46
	\pinlabel $d_1$ at 88 26	
	\pinlabel $q$ at 161 -5
	\pinlabel $d_1$ at 178 17
	\pinlabel $r_2$ at 152 36
		\pinlabel $v$ at 198 47
	\pinlabel $w$ at 155 72
	\pinlabel $E_q$ at 171 40
	\endlabellist
    \centering
    \includegraphics{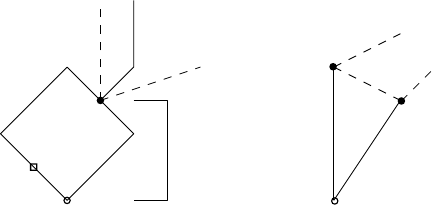}
    \caption{}
    \label{fig:CD}
\end{figure}

{The point $q$ sits on an internal edge, with the closer endpoint being a downfork and the further endpoint being an upfork in the Reeb graph $\Phi_{G}(q)$, as in Figure \ref{fig:CD}. Note that, generically speaking, the segment of $E_p$ from $w$ to $q$ must be a geodesic, otherwise there are two distinct geodesics from $w$ to $q$ that pass through $v$, and hence two distinct geodesics from $w$ to $v$, a nontrivial linear equality. Thus one of the geodesics from $w$ to $q$ consists of a segment of $E_p$ and the other passes through $v$, and hence any geodesic starting at $q$ passing through $w$ can be rerouted to a geodesic passing through $v$.\\
 
Consider the point $x$ that is on the other side of the loop $E_p$ from $u$. This produces a downfork in $\Phi_{G}(q)$, and hence a point in $\Psi_{G}(q)$ with birth time $d(q,u) + L(E_p)/2$. As we may assume that the geodesic from $q$ to this downfork can be rerouted through $v$, we have $d(q,u) + L(E_p)/2 = d_1 + d(v,u) + L(E_p)/2$. $\Psi_{G}(p) = \Psi_{G}(q)$ means that there is a similar birth time in $\Psi_{G}(p)$. This is strictly larger than $L(E_p)/2$, and hence the geodesic to the corresponding downfork passes through $u$. By Lemma \ref{lem:deathzero}, this is equal to $d_1 + \frac{1}{2}\sum_{i \in U} L(E_i)$ for some collection of edges in $G$ omitting $E_p$. This implies $d(v,u) + L(E_p)/2 = \frac{1}{2} \sum_{i \in U} L(E_i)$, a nontrivial linear equality among edge lengths.}\\

{\bf {Case D:}}

{Refer to Figure \ref{fig:DD}. The point $q$ sits on a self-loop distinct from the one on which $p$ lies.} The one-dimensional persistence diagram in $\Psi_{G}(p)$ contains the point $(L(E_p)/2,0)$, and similarly $\Psi_{G}(q)$ contains $(L(E_q),0)$. Lemma \ref{lem:detectvalence} states that these one-dimensional persistence diagrams contain unique points with death time zero, so $\Psi_{G}(p) = \Psi_{G}(q)$ implies $L(E_p) = L(E_q)$, a nontrivial linear equality as $q$ does not lie on $E_p$.
\end{proof}
\begin{figure}[htb]
	    	\labellist
	\small\hair 2pt
			\pinlabel $p$ at 33 -5
	\pinlabel $x$ at 8 17 
	\pinlabel $E_p$ at 24 29
	\pinlabel $u$ at 41 46
	\pinlabel $d_1$ at 88 26	
	\pinlabel $p$ at 156 -5
	\pinlabel $y$ at 131 17 
	\pinlabel $E_q$ at 147 29
	\pinlabel $v$ at 164 46
	\pinlabel $d_1$ at 211 26
	\endlabellist
	\centering
	\includegraphics{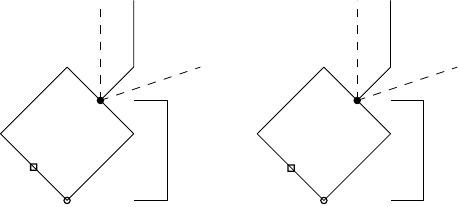}
	\caption{}
	\label{fig:DD}
\end{figure}

The upshot of the prior lemma is that, generically speaking, the failure of injectivity introduced by topological self-loops is local, occurring on the loop itself. Indeed, by lemma \ref{lem:d1}, the pairs of points on a topological self-loop that produce identical persistence diagrams are precisely those exchanged by the unique automorphism flipping the loop (this presumes, of course, that $G$ is not a circle). Figure \ref{fig:loopcollapse} demonstrates how a self-loop in $G$ becomes a leaf edge in $\mathcal{BT}(G)$. To reconstruct $G$, we detect that a failure of injectivity has occured -- the newly created leaf vertex in $\mathcal{BT}(G)$ cannot have come from a leaf vertex in $G$, as Lemma \ref{lem:detectvalence} tells us that such a leaf-vertex has no point in its one-dimensional persistence diagram with death time zero, whereas the leaf vertex in $\mathcal{BT}(G)$ is a barcode with just one such point, whose birth time is half the circumference of the original loop (note that persistence diagrams are labelled by dimension, so we can be certain to which dimension a particular point belongs). Our statement that the edge lengths of our graphs satisfy no nontrivial integer linear equalities implies that this failure of injectivity must come from a topological self-loop, and so we know how to reinsert the loop when reconstructing $G$. 
\begin{figure}[htb]
	    	\labellist
	\small\hair 2pt
	\pinlabel $G$ at 90 140
	\pinlabel $\mathcal{BT}(G)$ at 250 130
	\endlabellist
    \centering
    \includegraphics{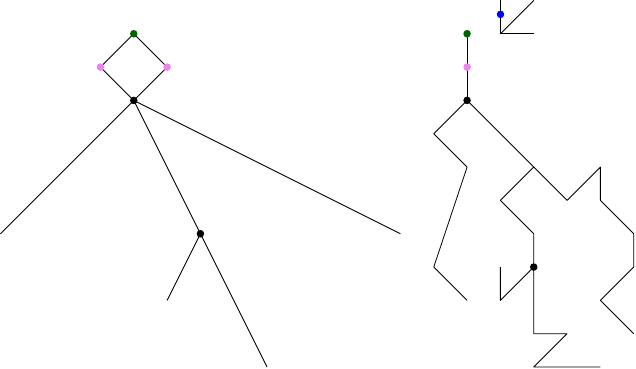}
    \caption{A self-loop in $G$ becomes a leaf edge in $\mathcal{BT}(G)$. The barcode at the end of this leaf edge contains a point of the form $(\cdot, 0)$ in one-dimensional persistence, informing us that it cannot have come from a leaf-vertex in $G$, but rather from a topological self-loop.}
    \label{fig:loopcollapse}
\end{figure}

The last possibility to consider is when there are fewer than three vertices in $G$ of valence distinct from two. If there are no such vertices then $G$ is a circle, whose persistence diagram is a single point. This point records the circumference of the circle, and as we have seen in the proof of Proposition \ref{birthdeathvary}, the only graphs producing one-point persistence diagrams are circles. Moving on, the only graphs with a single vertex of valence not equal to two are those arising as wedges of circles. The proof of Lemma \ref{lem:loopnonloop} demonstrates that the barcode transform of such a graph is a wedge of thorns, and by the method discussed above it is possible to detect this failure of injectivity and reconstruct the graph as well.\\

Next, the graphs with precisely two vertices of valence not equal to two are of the following form: two vertices $v$ and $w$ connected by a single edge or multiple edges, with some topological self-loops attached at $v$ and $w$, see Figure \ref{fig:twoverticesloops}. The results of Section \ref{sec: geninj} imply that the only failures of injectivity happen either at the self-loops or for pairs of points on the same edge connecting $v$ and $w$. Indeed, of the various proofs presented in Section \ref{sec: geninj}, only Proposition \ref{prop:vertexinj} (injectivity for vertices) and Lemma \ref{lem:sameedge} (injectivity for non-vertices on the same edge) require the existence of a third vertex. We claim that, when $G$ has topological self-loops, $\Psi_{G}$ is generically injective on the edges connecting $v$ and $w$. For suppose that $p$ and $q$ are two points on an edge $E$ connecting $v$ and $w$, with $\Psi_{G}(p) = \Psi_{G}(q)$. We know from Lemma \ref{lem:d1} that $p$ and $q$ are at the same distance $d_1$ from their closest vertices of valence at least three. Suppose, without loss of generality, that the points $p,q$ are arranged so that one encounters first $p$ and then $q$ when traveling from $v$ to $w$ along $E$, as in Figure \ref{fig:twoverticesloops}. We claim that a shortest geodesic from $p$ to one of the vertices $v,w$ must be the subsegment of $E$ connecting $p$ to $v$. Otherwise, the shortest geodesic is the subsegment of $E$ connecting $p$ to $w$, and hence, as there is a shorter sub-geodesic connecting $q$ to $w$, we find that $p$ and $q$ cannot be equidistant from their closest vertices of valence at least three. Symmetrically, a shortest geodesic from $q$ to one of $v,w$ must be the subsegment of $E$ connecting $q$ and $w$. Thus $d(p,v) = d(q,w) = d_1$. We claim further that $p$ is strictly closer to $v$ than $w$, as if $d(p,v) = d(p,w)$, the geodesic from $p$ to $w$ cannot pass through $v$ and passes instead through $q$, so that $d_1 = d(p,v) = d(p,w) > d(q,w) = d_1$, a contradiction. Similarly, $q$ is strictly closer to $w$ than $v$.
\begin{figure}[htb]
	    	\labellist
	\small\hair 2pt
	\pinlabel $v$ at 59 65
	\pinlabel $C$ at 36 76 
	\pinlabel $d_1$ at 77 104
	\pinlabel $p$ at 107 127
	\pinlabel $E$ at 140 145
	\pinlabel $q$ at 183 124 
	\pinlabel $d_1$ at 197 101
	\pinlabel $w$ at 200 65
		\pinlabel $C'$ at 224 68
	\pinlabel $z$ at 123 104
	\pinlabel $E'$ at 138 102
	\pinlabel $E''$ at 140 56
	\endlabellist
    \centering
    \includegraphics{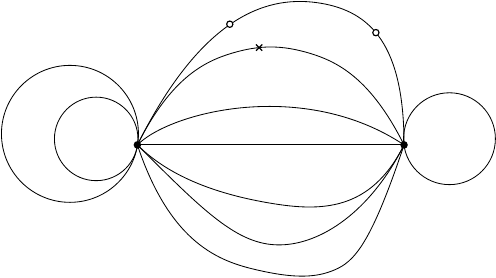}
    \caption{}
    \label{fig:twoverticesloops}
\end{figure}

 Suppose that at least one of $v$ or $w$ has an attached topological self-loop. Consider the shortest such loop (as two equal-length loops imply a nontrivial linear equality); without loss of generality, it is attached at $v$, and let us denote the loop edge by $C$. Then $\Psi_{G}(p)$ has a one-dimensional point born at $d_1 + L(C)/2$, and the same must be true of $\Psi_{G}(q)$. We now have four possibilities: the corresponding downfork in $\Phi_{G}(q)$ occurs either (i) at $v$, (ii) along a self-loop attached at $v$, (iii) along a self-loop attached at $w$, or (iv) at a downfork $z$ along an edge $E'$ connecting $v$ and $w$. Case (i) implies that $d_1 + L(C)/2$ is equal to $d(q,v)$. Moreover, if $v$ is a downfork for $q$, then symmetry (i.e. the vertex-flipping automorphism of the subgraph of $G$ obtained by removing self-loops) implies that $w$ is a downfork for $p$ {(as the downfork directions cannot be along self-loops, which move us further away from $p$ and $q$)}, so that there is a geodesic from $p$ to $w$ passing along an edge $E''$ connecting $v$ and $w$ that is distinct from $E$. Thus $d_1 + L(C)/2 = d(q,v) = d(p,w) = d_1 + L(E'')$, which is a nontrivial linear equality. Case (ii) is impossible as we have chosen $C$ to be the shortest possible loop edge, and $q$ is strictly further from the base of the loop than $p$. Case (iii) implies that $d_1 + L(C)/2 = d_1 + L(C')/2$, where $C'$ is a loop edge attached at $w$, giving a nontrivial linear equality. Lastly, case (iv) implies that $d_1 + L(C)/2 = \frac{1}{2}(L(E) + L(E'))$, another nontrivial linear equality.\\ 

Thus, if the failures of injectivity are those along topological self-loops, these are the only failures of injectivity, and we can reconstruct $G$ as before. The only other possibility is that $G$ consists of a pair of vertices joined by one or multiple edges, as on the left side of Figure \ref{fig:twovertexcollapse}.{ The shape of the resulting barcode transform, as seen in the right side of Figure \ref{fig:twovertexcollapse}, tells us that a failure of injectivity has occurred, as the barcode corresponding to the tip of the blue thorn tells us that the corresponding point has valence two. While the right-hand side looks identical to the barcode transform of a wedge of loops, the barcodes themselves are different. Consider again the blue thorn in the right-hand figure. If the original edge had length $L$, this blue thorn will have length $L/2$ in the intrinsic metric $\hat{d}_{B}$. However, the barcode at the tip of the blue thorn has no point of the form $(L/2,0)$, as would be the case for a self-loop. Indeed, the $(L/2)$-neighborhood of the corresponding point on the graph is just an edge, which is contractible and has no topology. Thus, when reconstructing $G$, we know to replace our thorns by edges and not self-loops.}
\begin{figure}[htb]
	    	\labellist
	\small\hair 2pt
	\pinlabel $L$ at 80 57
	\pinlabel $L/2$ at 208 86
	\endlabellist
    \centering
    \includegraphics{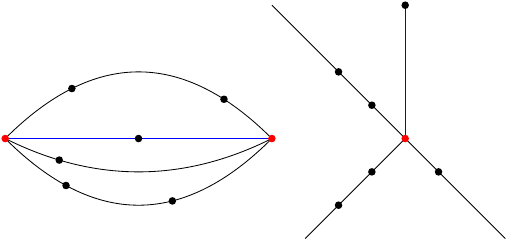}
    \caption{}
    \label{fig:twovertexcollapse}
\end{figure}
    
    Thus we have seen that, under our generic assumptions on the edge lengths of $G$, in all these cases when $\Psi_{G}$ fails to be injective, it is possible to detect that this has happened, that each resulting scenario can only arise in a unique way, and that there is a procedure to reconstruct $G$ from $\mathcal{BT}(G)$. 
    
\section{Theorem \ref{btlocalinj}: Local Injectivity}
\label{appbtlocalinj}

We will make use of a local injectivity result for Reeb graphs from \cite{localequiv} referenced in the background as Theorem \ref{ReebInequality}, setting $K=1/22$.\\

Let us first deal with the exceptional case when $G$ is a circle. We claim that if $G'$ is any other graph in the space $\mathbf{MGraphs}$ then $d_{PD}(G,G') > 0$. Observe first that the barcode transform of a circle of radius $c$ is the barcode consisting of the single interval $(0,c/2)$, and thus if $G'$ is a circle of a different circumference it cannot produce the same barcode. On the other hand, if $G'$ is not a circle then Proposition \ref{localisom} implies that $\mathcal{BT}(G')$ is not a single point, and hence it too cannot equal $\mathcal{BT}(G)$.\\ 

Next, let $G \in \mathbf{MGraphs}$ be any graph which is not a circle, and take any basepoint $x \in G$, giving rise to a barcode $\Psi_{G}(x)$. We will exhibit a constant $\epsilon(G,x) > 0$ such that if $G'$ is another metric graph with $0 < d(G,G') < \epsilon(G,x)$ then $\mathcal{BT}(G')$ omits the barcode $\Psi_{G}(x)$. Then since $\Psi_{G}$ and $\Psi_{G'}$ are continuous by Lemma \ref{psilip}, both $\mathcal{BT}(G)$ and $\mathcal{BT}(G')$ are compact subsets of Barcode space, and hence if they are not equal their Hausdorff distance is strictly positive: i.e. $d_{PD}(G,G') > 0$.\\

Firstly, let $a$ be the minimal distance between successive critical values for the Reeb graph $\Phi_{G}(x)$. Observe that Proposition \ref{localisom} implies that the basepoints which produce the same barcode as $x$ are isolated and hence the set of such basepoints, written $S_x$, is finite. Let $0 < r < a/32$. Let $\Omega_r = N_{r}(S_x)$ be the union of open neighborhoods of radius $r$ around points in $S$. Then $G \setminus \Omega_r$ is compact, and no point in this complement produces a barcode identical to $\Psi_{G}(x)$; continuity of $\Psi_{G}$ then implies that these barcodes are bounded away from $\Psi_{G}(x)$ by some constant $\delta_r > 0$.\\

Next, let $0 < \epsilon < \min (\delta_{r}/36, a/192)$. Suppose that $G'$ is another metric graph with $0 < d_{GH}(G,G') < \epsilon$, and take any $x' \in G'$. We will shot that $\Psi_{G'}(x') \neq \Psi_{G}(x)$. To see this, let $\mathcal{M}$ be any correspondence realizing the Gromov-Hausdorff distance between these two graphs, $\delta = d_{GH}(G,G')$; such a matching exists by the compactness of our graphs. Two cases emerge.\\

{\bf Case 1:} The correspondence $\mathcal{M}$ pairs $x'$ with a point $p \in \Omega_r$. at distance at most $r$ from some point $q \in S_x$.\\

{\bf Case 2:} The point $x'$ is paired with a point $p \in G \setminus \Omega_r$.\\

Let us first deal with case 1, and let $q \in S_x$ be a closest point to $p$, with $d(p,q) <r$. Since $\Psi_{G}(x) = \Psi_{G}(q)$, it will suffice to show that $\Psi_{G'}(x') \neq \Psi_{G}(q)$. We have seen in Lemma \ref{psilip} that the Reeb graphs $\Phi_{G}(q)$ and $\Phi_{G}(p)$ are within $r$ of each other in the FD distance. Moreover, because $p$ and $x'$ have been matched in $\mathcal{M}$, Theorem \ref{pairedgraphs} implies that their Reeb graphs are within 
\[6\delta < 6\epsilon < 6 \times \frac{a}{192} = \frac{a}{32}\]

of each other. Thus, applying the triangle inequality,

\[d_{FD}(\Phi_{G}(q),\Phi_{G'}(x')) < 2 \times \frac{a}{32} = \frac{a}{16}\]

Hence, applying Theorem \ref{ReebInequality}, we see that these two Reeb graphs can only produce identical barcodes if they are equal to each other. In that case, Corollary \ref{recovergraph} tells us that one can always recover the original metric graph from any of its Reeb graphs, so that $G' \simeq G'$, contradicting our assumption that $d_{GH}(G,G') > 0$.\\

Now let us consider case 2. By Theorem \ref{pairedgraphs}, the barcodes $\Psi_{G}(p)$ and $\Psi_{G'}(x')$ are within 
\[18 \delta < 18 \epsilon  < 18 \times \frac{\delta_r}{36} = \frac{\delta_r}{2}\]

of each other in the Bottleneck distance. But, as $p \in G \setminus \Omega_r$, the barcode $\Psi_{G}(p)$ is bounded away from $\Psi_{G}(x)$ by $\delta_r$. Applying the triangle inequality again, we find that $d_{B}(\Psi_{G}(x),\Psi_{G'}(x')) > \delta_r / 2 > 0$. $\square$

\section{Conclusion \& Discussion}

In this paper we reformulated and expanded upon the constructions of \cite{dey2015comparing} to analyze the inverse problem associated with the barcode transform, a complex homology-based invariant of metric structure, in the setting of metric graphs and metric measure graphs. We demonstrated that the barcode transform is locally injective on all of $\mathbf{MGraphs}$ and globally injective on a generic subset of $\mathbf{MGraphs}$. We saw, moreover, at the end of Section \ref{sec: selfloops}, that there are non-isometric pairs of graphs whose barcode transforms produce distinct but intrinsically isometric subsets of barcode space. Thus, there is strictly more information in $\mathcal{BT}(G)$ than its metric type as a subspace of $\mathbf{Barcodes}$. There are many possible directions for future research.

\begin{itemize}
    \item Extending the results of this paper to metric simplicial complexes or triangulated manifolds of higher dimension.
    \item Studying families of barcodes coming from other functions defined on graphs, such as heat kernels or diffusion distances.
    \item Incorporating this framework into the language of discrete Morse theory to improve computation and take advantage of the rich theory of Morse-type functions.
    \item More specifically, our injectivity results imply that the barcode transform often contains all the geometry of $G$. While the persistence distortion distance is one way of using the barcode transform to define a metric, there may be many others that take more advantage of this geometric information (e.g. by using Wasserstein distances instead of Bottleneck distances).
\end{itemize}

In terms of applications, the barcode transform can be approximated by sampling a finite collection of basepoint in $G$ and appealing to Theorem \ref{perstable}. In addition, the barcode transform can be composed with various kernels to define a transform taking values in a Hilbert space. This promises a framework for performing statistics and learning on spaces of metric graphs. Work needs to be done on determining which kernels are ideal and how effective the resulting transform can be made.

\renewcommand{\thesection}{\Alph{section}}
\setcounter{section}{0}

\section{Rips filtrations on geodesic trees}
\label{sec:Rips-trees}

\begin{lemma}\label{lem:Cech-Rips-hyper}
Let $(X,d_X)$ be a geodesic tree. Then, the \v Cech and Rips filtrations on $X$ are equal up to a factor of $2$ in the parameter, that is:
\[
\forall t\geq 0,\ \mathrm{C}_t(X, d_X) = \mathrm{R}_{2t}(X, d_X).
\]
\end{lemma}
\begin{proof}
It suffices to show the following statement:
\[
\forall n\geq 2,\ \forall X_0, \cdots, X_n \subseteq  X\ \mbox{convex},\ \left[X_i\cap X_j \neq \emptyset\quad \forall\, 0\leq i\leq j\leq n\right] \Longleftrightarrow \left[ \bigcap_{i=0}^n X_i \neq \emptyset\right].
\]
The proof is by recurrence. Note that the direction $\Leftarrow$ is trivial, so we show the direction $\Rightarrow$.

The base case ($n=2$) follows from the fact that every geodesic triangle in the tree $X$ is a tripod, $X$ being a $0$-hyperbolic space in the sense of Gromov---see e.g. Section~1 in Chapter~III.H of~\cite{bridson2013metric}. The details are as follows: pick a point $x_{ij}\in X_i\cap X_j$ for every $i,j$, then form the (unique) geodesic triangle $[x_{01}, x_{12}, x_{20}]$. Since the triangle is a tripod, there is a point $p$ at the intersection between the three geodesics $[x_{01}, x_{12}]$, $[x_{12}, x_{20}]$, $[x_{20}, x_{01}]$. Now, for every $i,j,k$ we have $[x_{ij}, x_{jk}]\subseteq X_j$ by convexity, and so $p\in X_0\cap X_1 \cap X_2$. This concludes the base case. 

For the recurrence step, assume the result holds for some $(n-1)\geq 2$ and let us prove it for $n$. Take $X_0, \cdots, X_n\subseteq X$ convex such that $X_i\cap X_j \neq\emptyset$ for all $0\leq i\leq j \leq n$. Then, the base case implies that $X_i\cap X_j\cap X_n\neq \emptyset$ for all $0\leq i\leq j \leq n-1$. Moreover, the sets $X_i\cap X_j$ are convex in $X_n$, which itself is a geodesic tree (as a convex subset of the tree $X$).  Hence, the case $(n-1)$ implies that $\bigcap_{i=0}^{n-1} (X_i\cap X_n)\neq \emptyset$, and so $\bigcap_{i=0}^n X_i \neq \emptyset$. This concludes the recurrence, and thereby the proof of the lemma. 
\end{proof}

\bibliographystyle{plain}
\bibliography{mybib}

\begin{bibdiv}
\begin{biblist}

\bib{burago2001course}{book}{
      author={Burago, Dmitri},
      author={Burago, Yuri},
      author={Ivanov, Sergei},
       title={A course in metric geometry},
   publisher={American Mathematical Society Providence},
        date={2001},
      volume={33},
}

\bib{reebgraphdistance}{inproceedings}{
      author={Bauer, Ulrich},
      author={Ge, Xiaoyin},
      author={Wang, Yusu},
       title={Measuring distance between reeb graphs},
organization={ACM},
        date={2014},
   booktitle={Proceedings of the thirtieth annual symposium on computational
  geometry},
       pages={464},
}

\bib{bridson2013metric}{book}{
      author={Bridson, Martin~R},
      author={Haefliger, Andr{\'e}},
       title={Metric spaces of non-positive curvature},
   publisher={Springer Science \& Business Media},
        date={2013},
      volume={319},
}

\bib{chazal2012structure}{book}{
      author={Chazal, Fr{\'e}d{\'e}ric},
      author={De~Silva, Vin},
      author={Glisse, Marc},
      author={Oudot, Steve},
       title={The structure and stability of persistence modules},
   publisher={Springer},
        date={2016},
}

\bib{chazal2014persistence}{article}{
      author={Chazal, Fr{\'e}d{\'e}ric},
      author={De~Silva, Vin},
      author={Oudot, Steve},
       title={Persistence stability for geometric complexes},
        date={2014},
     journal={Geometriae Dedicata},
      volume={173},
      number={1},
       pages={193\ndash 214},
}

\bib{localequiv}{inproceedings}{
      author={Carri{\`e}re, Mathieu},
      author={Oudot, Steve},
       title={{Local Equivalence and Intrinsic Metrics between Reeb Graphs}},
        date={2017},
   booktitle={33rd international symposium on computational geometry (socg
  2017)},
      editor={Aronov, Boris},
      editor={Katz, Matthew~J.},
      series={Leibniz International Proceedings in Informatics (LIPIcs)},
      volume={77},
   publisher={Schloss Dagstuhl--Leibniz-Zentrum fuer Informatik},
     address={Dagstuhl, Germany},
       pages={25:1\ndash 25:15},
         url={http://drops.dagstuhl.de/opus/volltexte/2017/7179},
}

\bib{cohen2009extending}{article}{
      author={Cohen-Steiner, David},
      author={Edelsbrunner, Herbert},
      author={Harer, John},
       title={Extending persistence using poincar{\'e} and lefschetz duality},
        date={2009},
     journal={Foundations of Computational Mathematics},
      volume={9},
      number={1},
       pages={79\ndash 103},
}

\bib{curry2017fiber}{article}{
      author={Curry, Justin},
       title={The fiber of the persistence map for functions on the interval},
        date={2018},
     journal={Journal of Applied and Computational Topology},
      volume={2},
      number={3-4},
       pages={301\ndash 321},
}

\bib{dey2015comparing}{inproceedings}{
      author={Dey, Tamal~K.},
      author={Shi, Dayu},
      author={Wang, Yusu},
       title={Comparing graphs via persistence distortion},
        date={2015},
   booktitle={31st international symposium on computational geometry, socg
  2015, june 22-25, 2015, eindhoven, the netherlands},
      editor={Arge, Lars},
      editor={Pach, J{\'{a}}nos},
      series={LIPIcs},
      volume={34},
   publisher={Schloss Dagstuhl - Leibniz-Zentrum f{\"{u}}r Informatik},
       pages={491\ndash 506},
         url={https://doi.org/10.4230/LIPIcs.SOCG.2015.491},
}

\bib{gasparovic2017complete}{incollection}{
      author={Gasparovic, Ellen},
      author={Gommel, Maria},
      author={Purvine, Emilie},
      author={Sazdanovic, Radmila},
      author={Wang, Bei},
      author={Wang, Yusu},
      author={Ziegelmeier, Lori},
       title={A complete characterization of the one-dimensional intrinsic
  {\v{c}}ech persistence diagrams for metric graphs},
        date={2018},
   booktitle={Research in computational topology},
   publisher={Springer},
       pages={33\ndash 56},
}

\bib{gameiro2016continuation}{article}{
      author={Gameiro, Marcio},
      author={Hiraoka, Yasuaki},
      author={Obayashi, Ippei},
       title={Continuation of point clouds via persistence diagrams},
        date={2016},
     journal={Physica D: Nonlinear Phenomena},
      volume={334},
       pages={118\ndash 132},
}

\bib{ivanov2016realizations}{article}{
      author={Ivanov, Alexander},
      author={Iliadis, Stavros},
      author={Tuzhilin, Alexey},
       title={Realizations of gromov-hausdorff distance},
        date={2016},
     journal={arXiv preprint arXiv:1603.08850},
}

\bib{kerber2017geometry}{article}{
      author={Kerber, Michael},
      author={Morozov, Dmitriy},
      author={Nigmetov, Arnur},
       title={Geometry helps to compare persistence diagrams},
        date={2017},
     journal={Journal of Experimental Algorithmics (JEA)},
      volume={22},
      number={1},
       pages={1\ndash 4},
}

\bib{memoli2011gromov}{article}{
      author={M{\'e}moli, Facundo},
       title={Gromov--{W}asserstein distances and the metric approach to object
  matching},
        date={2011},
     journal={Foundations of computational mathematics},
      volume={11},
      number={4},
       pages={417\ndash 487},
}

\bib{turner2014persistent}{article}{
      author={Turner, Katharine},
      author={Mukherjee, Sayan},
      author={Boyer, Doug~M},
       title={Persistent homology transform for modeling shapes and surfaces},
        date={2014},
     journal={Information and Inference: A Journal of the IMA},
      volume={3},
      number={4},
       pages={310\ndash 344},
}

\end{biblist}
\end{bibdiv}

\end{document}